\documentclass{article}
\pdfoutput=1
\usepackage{amsmath,amssymb,verbatim,amsthm}
\usepackage{amsopn}
\usepackage{enumitem}
\usepackage[mathscr]{eucal}
\usepackage{endnotes}
\usepackage{graphicx}
\usepackage{color}
\usepackage[normalem]{ulem}

\def\R{\mathbb R}
\def\C{\mathbb C}

\def\trace{\operatorname{tr}}

\def\Ad{\operatorname{Ad}}

\def\w{\omega}

\def\su{\mathfrak{su}}

 \newtheorem{thm}{Theorem}[section]
 \newtheorem{cor}[thm]{Corollary}
 \newtheorem{lem}[thm]{Lemma}
 \newtheorem{prop}[thm]{Proposition}
 \theoremstyle{definition}
 
 \newtheorem{defn}[thm]{Definition}
 \theoremstyle{remark}
 \newtheorem{rem}[thm]{Remark}
 
 \numberwithin{equation}{section}
 \newtheorem{prethought}[thm]{Thought}
 \newtheorem{needsfixing}{Needsfixing}
 
 \newtheorem{example}[thm]{Example}

\def\f{\mathsf f}

\begin{document}
\title{Pseudo-spherical Surfaces of Low Differentiability}
\author{Josef F. Dorfmeister \thanks{Zentrum Mathematik, Technische Universit\"{a}t M\"{u}nchen, D-85747 Garching bei M\"{u}nchen, Germany, dorfm@ma.tum.de} \and 
Ivan Sterling \thanks{Mathematics and Computer Science Department, St Mary's College of Maryland, St Mary's City, MD 20686-3001, USA, isterling@smcm.edu}}
\maketitle
\begin{abstract}\noindent We continue our investigations into Toda's algorithm \cite{T,DIS}; a Weierstrass-type representation of Gauss curvature $K=-1$ surfaces in $\mathbb{R}^3$.  We show that $C^0$ input potentials correspond in an appealing way to a special new class of surfaces, with $K=-1$, which we call $C^{1M}$.  These are surfaces which may not be $C^2$, but whose mixed second partials are continuous and equal.  We also extend several results of Hartman-Wintner \cite{HW} concerning special coordinate changes which increase differentiability of immersions of $K=-1$ surfaces.  We prove a $C^{1M}$ version of Hilbert's Theorem.
\end{abstract}

\section{Introduction}
We continue our investigations into Toda's algorithm \cite{T,DIS}; a Weierstrass-type representation of Gauss curvature $K=-1$ surfaces in $\mathbb{R}^3$.  Briefly the algorithm is as follows (see Section \ref{todaalg} for a precise description):   Assume $\lambda \in \mathbb{R}^+$ is the loop parameter.  Input two loops $\eta = (\eta_+(x),\eta_-(y))$, whose order of differentiability will be a central topic of this paper.  Solve the system of loop ODE's
  \begin{eqnarray*}
  \frac{\partial U_+}{\partial x} & = & U_+ \eta_+, \\
  \frac{\partial U_-}{\partial y} & = & U_- \eta_-.
  \end{eqnarray*}
Then apply Birkhoff's Factorization to $U_-^{-1} U_+$ (which exists globally without singularities by Brander \cite{B}):
 \[U_-^{-1} U_+ = L_+L_-^{-1}.\] 
Let $\widehat{U}:=U_- L_+ = U_+ L_-$.  Then $\widehat{U}$ satisfies a system of the form
  \begin{eqnarray*}
\frac{\partial \widehat{U}}{\partial x} & = &  \frac{i}{2} \widehat{U} 
\left[\begin{array}{cc} -\phi_x  & \lambda \\ \lambda & \phi_x \end{array}\right], \\
  \frac{\partial \widehat{U}}{\partial y} & = &  -\frac{i}{2 \lambda} \widehat{U} 
\left[\begin{array}{cc} 0  & e^{i \phi} \\  e^{-i \phi} & 0 \end{array}\right],
 \end{eqnarray*}
where $\phi$ satisfies the sine-Gordan equation $\phi_{xy}=\sin \phi$ (in some of our cases, only distributionally).
Finally set \[f_\eta=\frac{\partial}{\partial t} \hat{U} \hat{U}^{-1} \Big|_{t=0},\] where $\lambda = e^t$.  Then $f_\eta$ is a surface (whose order of differentiability is studied here) in $\mathbb{R}^3$ with $\cos \phi = <f_{\eta_x},f_{\eta_y}>$ and $K=-1$.

Let $D_{(x,y)}$ be a simply connected open set in $\mathbb{R}^2$ (whenever necessary we will assume $D = J \times J$ for some interval $J$).  A continuously differentiable map $f:D_{(x,y)} \stackrel{C^1}{\longrightarrow} \mathbb{R}^3$ is called a regular immersion at $p$ if $rank_f(p)=2$ (and regular on $D$ if it is regular at every $p \in D$).  Regularity is independent of coordinates. $f$  is called weakly-regular at $p$, with respect to $(x,y)$, if ${<f_x,f_x>\;\neq 0}$ and $<f_y,f_y>\; \neq 0$.  Weakly-regular at $p$ implies $rank_f(p) \geq 1$ , but is actually much stronger. It is important to note that weak-regularity is dependent on coordinates.  $f$ is called Chebyshev at $p$, with respect to $(x,y)$, if  ${<f_x,f_x> = 1}$ and $<f_y,f_y> = 1$.

Let $\hat{f}:D_{(u,v)} \stackrel{C^2}{\longrightarrow} \mathbb{R}^3$ be a twice continuously differentiable regular immersion with parameters $u,v$, and negative Gaussian curvature $K$.  Then at every point $\hat{f}(u,v)$ there is a distinct pair of directions, called the pair of asymptotic directions, characterized by the vanishing of the second fundamental form.  By an asymptotic curve on $\hat{f}$ is meant a curve along which the direction of the tangent is always in an asymptotic direction.  In particular a reparametrization $f=\hat{f} \circ T$, of $\hat{f}$, by a change of coordinates $T:D_{(x,y)} \longrightarrow D_{(u,v)}$, is called a parametrization by asymptotic coordinates if all parameter curves are asymptotic curves.

It is well known that if  $\hat{f}:D \stackrel{C^4}{\longrightarrow} \mathbb{R}^3$ is a four times continuously differentiable regular immersion with $K=-1$, then $\hat{f}$ can be reparametrized by a $C^3$-diffeomorphism $T$ to an $f=\hat{f} \circ T$  asymptotic Chebyshev immersion with $K=-1.$   Furthermore all asymptotic Chebyshev immersions $f:D \stackrel{C^3}{\longrightarrow} \mathbb{R}^3$ with $K=-1$ arise by Toda's \cite{T} loop group algorithm from $C^2$ potentials $\eta$.  Conversely, given a $C^2$ potential $\eta$, Toda's algorithm produces a (possibly weakly-regular) $C^3$ asymptotic Chebyshev immersion with $K=-1$.  Moreover, see \cite{HW}, by a reparametrization one can always (locally) obtain a $C^4$ immersion $\hat{f}=f \circ S$  in graph coordinates where $S:D_{(u,v)} \longrightarrow D_{(x,y)}$.

We can improve these results by two degrees of differentiability.  We first define $C^{1M}$ functions (functions that are $C^1$ and whose mixed partials exist, are equal and are continuous).  Given a regular  $\hat{f}:D \stackrel{C^2}{\longrightarrow} \mathbb{R}^3$ which satisfies $K=-1$, then Hartman-Wintner \cite{HW} proved $\hat{f}$ can be $C^{1M}$-reparametrized by asymptotic Chebyshev coordinates.  Furthermore we prove in this paper that all such immersions arise by Toda's \cite{T} loop group algorithm from $C^0$-potentials.  Conversely, given a $C^0$-potential $\eta$, Toda's algorithm produces a (possibly weakly-regular) $C^{1M}$ asymptotic Chebyshev immersion $f$ with $K=-1$.  Moreover, assuming $f$ is regular, using the method of \cite{HW} we prove that by a reparametrization one can always (locally) obtain from such an $f$ a $C^2$-immersion $\hat{f}=f \circ S$ in graph coordinates.  Finally, we are able to patch these local reparametrizations $S$ together to obtain a global reparametization $\rho$ such that $f_{new} = f \circ \rho$ is a global $C^2$-immersion.

For a regular immersion
\[f:D_{(x,y)} \stackrel{C^3}{\longrightarrow} \mathbb{R}^3,\] 
by asymptotic Chebyshev coordinates the following are known to be equivalent.
\begin{enumerate}
\item The immersion has constant negative Gauss curvature.
\item The angle $\omega$ between the asymptotic lines is never $0$ (or $\pi$) and $\omega_{xy}=\sin \omega$.  
\item \label{BelEnn}  The asymptotic curves are of constant torsion.
\item The Gauss map $N$ is Lorentz harmonic and $N_{xy}= \cos{\omega} \;N$, where $\omega$ denotes the angle between asymptotic lines.
\item There exists some $C^2$ input potential $\eta$ (in the sense of Toda \cite{T}) such that $f=f_\eta$.
\end{enumerate}

The converse of item two is also true, therefore from the PDE point of view we are trying to solve the sine-Gordon equation
\begin{equation}\label{sgord} \phi_{xy} = \sin{\phi}. \end{equation}
Equation (\ref{sgord}) is the integrability condition arising in Toda's loop construction.  That solutions to the sine-Gordon equation correspond to families of $K=-1$ surfaces has been known since the 1840's.

In \cite{DIS} we showed that asymptotic Chebyshev immersions $f:D_{(x,y)} \stackrel{C^n}{\longrightarrow} \mathbb{R}^3$, $n \geq 3$ with $K=-1$, arise in the loop group approach from potentials $\eta$  of type $C^{n-1}$.  Conversely we showed that potentials $\eta$  of type $C^{n-1}$, $n \geq 3$, give rise to (possibly weakly-regular)  asymptotic Chebyshev immersions $f:D_{(x,y)} \stackrel{C^n}{\longrightarrow} \mathbb{R}^3$, $n \geq 3$ with $K=-1$.   Here we will consider the cases n=1, 2 and thus the case of input potentials  $\eta$ of type $C^0$ and $C^1$ and the maps $f_\eta$ they generate, even though much of the case of $n=2$ is already contained in \cite{DIS}.  For the cases $n=1$ and $n=2$ we will revisit and discuss in detail aspects of the five equivalent properties listed above.  That is the main goal of this paper.

There is an extensive body of work on surfaces of negative curvature  (including those of constant negative curvature) of low differentiability (see \cite{R},\cite{BS} and their extensive references).  There is also an extensive literature going back over one hundred years for discrete pseudo-spherical surfaces (see \cite{HPS} for some references) that complements and often informs the $C^r, \; r \leq \omega,$ theory.


We define a weakly-regular immersion $N:D  \stackrel{C^{1M}}{\longrightarrow} \mathbb{S}^2$  to be weakly-(Lorentz) harmonic if there exist a function $h:D \longrightarrow \mathbb{R}$ such that $N_{xy}(p) = N_{yx}(p) = h(p) N(p)$ for all $p \in D$.  Then a weakly-regular immersion $f:D \stackrel{C^{1M}}{\longrightarrow} \mathbb{R}^3$ is called a \textit{ps-front} if there exist a weakly-harmonic $N:D  \stackrel{C^{1M}}{\longrightarrow} \mathbb{S}^2$  such that $f_x=N \times N_x$, $f_y= -N \times N_y$.  This is a special case of a front, as defined in \cite{SUY}.

The outline of the paper is as follows.  In Section \ref{todaalg} we start with a $C^0$ input potential $\eta$ and apply Toda's loop group algorithm to construct a   weakly-regular asymptotic Chebyshev ps-front $f_\eta:D \stackrel{C^{1M}}{\longrightarrow} \mathbb{R}^3$. We show that any such map has a second fundamental form with $K=-1$ and that the angle $\omega$ between the asymptotic lines satisfies a version of the sine-Gordon equation $\omega_{xy}=\sin \omega$.  In Section \ref{converse} we show the converse, that for any such $f$ there exists a $C^0$ input potential $\eta$ such that $f=f_\eta$. In Section \ref{hwthy} we carry out the Hartman-Wintner \cite{HW} theory as discussed above. We prove a $C^{1M}$ version of Hilbert's Theorem in Section \ref{hilbert}.  In Section \ref{cusppics} we show, by way of the PS-sphere how our frame and fronts manage to behave well, even along a singular cusp line.  

\section{From $C^0$ Potentials to $C^{1M}$-Immersions} \label{todaalg}
\subsection{$C^{1M}$ functions}
\begin{defn} \label{defnregs1}Let $g:D \stackrel{C^1}{\longrightarrow} \mathbb{R}$.  Then
$g$ is called $C^{1M}$ at $p \in D$ if $g_{xy}$ and $g_{yx}$ exist, are continuous, and are equal at $p$.
Let $f:D \stackrel{C^1}{\longrightarrow} \mathbb{R}^3$.  Then $f$ is called $C^{1M}$ at $p \in D$ if all its components are.
\end{defn}
\begin{example} \label{obvious}
$f(x,y) = x^{\frac{4}{3}}y^{\frac{4}{3}}$ is $C^{1M}$ everywhere,  but is not $C^{2}$ at the origin.
\end{example}  
\begin{example}
The $C^{1M}$ condition is not invariant under a regular change of coordinates.  For example if we let $x= u+v$ and $y=u-v$ in Example \ref{obvious}, then it is still $C^1$ everywhere, but not $C^{1M}$ with respect to $u,v$ at the origin.
\end{example}
\begin{example}The function 
\[f(x,y) = \int_{-1}^x \int_y^\infty \frac{cos(tw) -1}{w^2} dw dt \]
is $C^{1M}$ on $\mathbb{R}^2$, but is not $C^2$ along a line.
\end{example}

 \subsection{The Curvature of Regular $C^{1M}$-Immersions} \label{kimm}
We point out the precise conditions required for familiar definitions for immersions.
\begin{defn}
If $f:D_{(x,y)} \stackrel{C^1}{\longrightarrow} \mathbb{R}^3$, then $E:=<f_x,f_x>$, $F:=<f_x,f_y>$, and $G=<f_y,f_y>$.  If moreover $N$ is $C^1$, then we define $\ell:=-<f_x,N_x>$ and $n:=-<f_y,N_y>$.  If furthermore $f$ is $C^{1M}$, then we define ${m:=<f_{xy},N>}={<f_{yx},N>}=-<f_x,N_y>=-<f_y,N_x>$ and finally if $f$ is regular we define the Gauss curvature of such an immersion by
\[K:=\frac{\ell n - m^2}{EG-F^2}.\]
\end{defn}
\noindent These are natural definitions perfectly suited to the $C^{1M}$-immersion setting.  We will need the expected result that under reasonable conditions the Gauss curvature of an immersion is independent of coordinates.
\begin{lem}
If $f:D_{(x,y)} \stackrel{C^{1M}}{\longrightarrow} \mathbb{R}^3$ is a regular immersion, with $N$ $C^1$ and $T:D_{(u,v)} \stackrel{C^{1}}{\longrightarrow} D_{(x,y)}$ a diffeomorphism with $\tilde{f}:=f \circ T$ also $C^{1M}$, then $\tilde{K}=K$.
\end{lem}
\begin{proof}
Note $\tilde{N} = N \circ T$ and use the chain rule.
\end{proof}

\subsection{Loops} 
We will be dealing with 1-forms (for example $\eta_-$ and $\eta_+$ below) on $D$ taking values in the loop algebra
\begin{equation*}\label{algloopdef}
\Lambda_{\sigma}\su(2) = \{ X: \R^* \to \su(2) |\, X(-\lambda) = \Ad(\sigma_3)\cdot X(\lambda) \},
\text{where } \sigma_3 = \left( \begin{smallmatrix} 1 & 0 \\ 0 & -1\end{smallmatrix}\right).
\end{equation*}
Likewise, we will have maps (for example $U+$ and $U_-$ below)  from $D$ to the loop group
$$\Lambda_{\sigma} SU(2) = \{ g: \R^* \to SU(2) | \,g(-\lambda) = \Ad(\sigma_3)\cdot g(\lambda) \}.$$
Loops satisfying the $\Ad(\sigma_3)$ condition are sometimes referred to as {\em twisted}.
We will be specifically interested in those subgroups, denoted by $\Lambda_\sigma \su(2)$ and $\widetilde{\Lambda}_\sigma SU(2)$
respectively, consisting of loops which extend to $\C^*$ as  analytic functions of $\lambda$.
(Note, however, that such extensions will take values in $\mathfrak{sl}(2,\C)$ and $SL(2,\C)$ respectively.)
In fact, the goal of the method is to recover such loops from analytic data specified along
a pair of characteristic curves in $D$.

\noindent Within the group of loops that extend analytically to $\C^*$, we define subgroups of loops which extend
to $\lambda=0$ or $\lambda=\infty$:
\begin{align*}
\Lambda_{\sigma}^+ SU(2) &= \{ g \in \widetilde{\Lambda}_\sigma SU(2) | g = g_0 + \lambda g_1 + \lambda^2 g_2 + \ldots \} \\
\Lambda_{\sigma}^-SU(2) &=\{ g \in \widetilde{\Lambda}_\sigma SU(2) | g = g_0 + \lambda^{-1} g_1 + \lambda^{-2} g_2 + \ldots \}
\end{align*}
Within these, we let $\Lambda_{\sigma_*}^+ SU(2)$ and $\Lambda_{\sigma_*}^- SU(2)$
be the subgroups of loops where $g_0$ is the identity matrix.  A key tool we will use is
\begin{thm}[Birkhoff Decomposition (Brander \cite{B}, Toda \cite{T})]
The multiplication maps
$$\Lambda_{\sigma_*}^+ SU(2) \times \Lambda_{\sigma}^- SU(2) \to \widetilde{\Lambda}_\sigma SU(2), \qquad
\Lambda_{\sigma_*}^- SU(2) \times \Lambda_{\sigma}^+ SU(2) \to \widetilde{\Lambda}_\sigma SU(2)$$
are diffeomorphisms.
\end{thm}

\begin{rem} In general, the Birkhoff decomposition theorem asserts that
the multiplication maps are analytic diffeomorphisms onto an open dense subset, known as the {\bf big cell}.  However, it follows from the result of Brander \cite{B} that in the case of compact semisimple Lie groups like $SU(2)$, the big cell is everything.
\end{rem}

\begin{rem}
We will make repeated use of the fact that an element $g_+$ of $\Lambda_{\sigma}^+ SU(2)$, $g_+ = g_0 + \lambda g_1 + \lambda^2 g_2 + \ldots$ has for even labels $g_{2k}$ diagonal matrices of the form $g_{2k}=\left(\begin{array}{cc} a_{2k} & 0 \\ 0 & \overline{a_{2k}} \end{array}\right)$ and for odd labels
$g_{2k+1} = \left(\begin{array}{cc} 0 & b_{2k+1} \\  -\overline{b_{2k+1}}  & 0 \end{array}\right)$.  For example, we have 
\begin{eqnarray*}
1 &=& \det g_+ = \left(\begin{array}{cc}a_0+\lambda^2 a_2 & \lambda b_1 \\
-\lambda \bar{b_1} & \overline{a_0}+\lambda^2 \overline{a_2} \end{array}\right) \mod \lambda^3 \\
&=& |a_0|^2 + (a_0 \overline{a_2} + \overline{a_0} a_2) \lambda^2 + \lambda^2 |b_1|^2 + \cdots
\end{eqnarray*}
Hence $|a_0|=1$ and $a_0 \overline{a_2} + \overline{a_0} a_2= -|b_1|^2$.  Of course, the analogous remark would hold for $g_- \in \Lambda_{\sigma}^- SU(2).$
\end{rem}

\subsection{Potential to Frame} 
Let the input potential
\begin{equation} \eta  = \big{(}\eta_-(y) dy, \eta_+(x) dx \big{)}  \end{equation}
be a pair of matrices of the form
\begin{eqnarray}
\eta_-(y) & = & -\frac{i}{2} \lambda^{-1} 
\left(\begin{array}{cc}
0 & e^{i \beta(y)} \\
e^{-i \beta(y)} & 0 
\end{array} \right), \\
\eta_+(x) & = &  \frac{i}{2} \lambda
\left(\begin{array}{cc}
0 & e^{-i \alpha(x)} \\
e^{i \alpha(x)} & 0 
\end{array} \right), 
\end{eqnarray}
where $\alpha, \beta$ are $C^0$ functions which are defined on some open intervals $J_-$ and $J_+$  respectively, and where $\lambda > 0$.  For simplicity of notation we will assume without loss of generality $0 \in J_\pm$.  According to the construction procedure of \cite{DIS} we solve next the ODEs:
\begin{alignat}{3}
U_{+_x} & =  U_+ \eta_+,  & \;\; U_+(0,\lambda) & =I, \\
U_{-_y} & =  U_- \eta_-,  & \;\; U_-(0,\lambda) & =I.
\end{alignat}

\begin{rem}
$U_+(x,\lambda)$ is $C^1$ in $x$ and $U_-(y,\lambda)$ is $C^1$ in $y$.  Both have holomorphic extensions to $\lambda \in \mathbb{C}^*$.
\end{rem}
\noindent The next step in the procedure of \cite{DIS} is the Birkhoff splitting
\begin{equation} \label{birk}
U_-(y,\lambda)^{-1}U_+(x,\lambda) = L_+(x,y,\lambda) L_-(x,y,\lambda)^{-1},
\end{equation}
where
\begin{eqnarray}
L_+ &=& L_{+_0} + \lambda L_{+_1} + \cdots, \\
L_- &=& I + \lambda^{-1} L_{-_{-1}} + \cdots.
\end{eqnarray}
We set 
\begin{equation} \label{mainUsplit}
\widehat{U} := U_-L_+=U_+L_-.
\end{equation}
\begin{rem}
$L_+,L_-,\widehat{U}$ are $C^1$ in $(x,y)$.  This follows from the fact that the left side of (\ref{birk}) is $C^1$ in $(x,y)$ and, by a result of Brander \cite{B}, the Birkhoff splitting is global and analytic in the coefficients of $U_-^{-1}U_+$.
\end{rem}
\begin{lem} \label{Uform}
Under the assumptions of this section we obtain
\begin{enumerate}[label={\rm{\alph*)}},ref={\alph*)}]
\item \label{Uforma} \mbox{} \vspace{-.3in} \begin{eqnarray}
\widehat{U}(0,0,\lambda) &=& I \label{birkform1}
\end{eqnarray}
\item \label{Uformb} \mbox{} \vspace{-.3in} \begin{eqnarray}
\widehat{U}^{-1} \widehat{U}_y &=& -\frac{1}{2} \lambda^{-1} \left( \begin{array}{cc} 0 & p \\ -\bar{p} & 0 \end{array} \right),  \label{birkform2}
\end{eqnarray}
\item \label{Uformc} \mbox{} \vspace{-.3in} \begin{eqnarray}
\widehat{U}^{-1} \widehat{U}_x &=& \frac{1}{2} \left(\begin{array}{cc} i r & \lambda q \\ -\lambda \bar{q} & -i r \end{array}\right).  \label{birkform3}
\end{eqnarray}
\end{enumerate}
with $C^0$ functions $p,q, r$  and $r$ real.
\end{lem}
\begin{proof}
\ref{Uforma} The left side of (\ref{birk}) is $I$ at $(0,0)$.  The splitting of the right side of (\ref{birk}) is unique.\\
\ref{Uformc} The equation $\widehat{U}=U_-L_+$ implies 
\[\widehat{U}^{-1} \widehat{U}_x = L_+^{-1} L_{+_x} = \frac{1}{2} \left(\begin{array}{cc} i r & \lambda q \\ -\lambda \bar{q} & -i r \end{array}\right).\]  
\ref{Uformb} From $\widehat{U}=U_+L_-$ we obtain $\widehat{U}^{-1}\widehat{U}_y = L_-^{-1} L_{-_y}$.  This has the form stated.
\end{proof}
\begin{rem} We will see below that $p,q,r$ have a very specific form and higher degrees of partial differentiability.
\end{rem}
\noindent We want to relate the matrix entries of the Lemma \ref{Uform} more closely to $\eta_+,\eta_-$.  Setting $x=0$ and $y=0$ respectively in (\ref{mainUsplit}) we obtain
\begin{eqnarray*}
\widehat{U}(x,0,\lambda) &=& U_-(0,\lambda) L_+(x,0,\lambda) = U_+(x,\lambda) L_-(x,0,\lambda),  \\
\widehat{U}(0,y,\lambda) &=& U_-(y,\lambda) L_+(0,y,\lambda) = U_+(0,\lambda) L_-(0,y,\lambda),
\end{eqnarray*}
and therefore
\begin{eqnarray*}
\lefteqn{\frac{1}{2}\left(\begin{array}{cc}i r(x,0) & \lambda q(x,0) \\ 
-\lambda \overline{q(x,0)} & -i r(x,0) \end{array}\right)}   \\
& = &L_-^{-1}(x,0,\lambda) \eta_+(x) L_-(x,0,\lambda) + L_-^{-1}(x,0,\lambda) L_{-_x}(x,0,\lambda), \\
\lefteqn{-\frac{1}{2} \lambda^{-1} \left(\begin{array}{cc} 0 & p(0,y) \\
-\overline{p(0,y)} & 0 \end{array}\right)}   \\
 &=&L_+^{-1}(0,y,\lambda) \eta_-(y) L_+(0,y,\lambda) + L_+^{-1}(0,y,\lambda) L_{+_y}(x,0,\lambda) .
\end{eqnarray*}
From (\ref{birk}) we obtain by the uniqueness of the Birkhoff splitting
\begin{eqnarray*}
L_-(x,0,\lambda)&=& I,\\
L_+(0,y,\lambda) & = & I,
\end{eqnarray*}
hence
\[ r(x,0)=0, q(x,0) = i e^{-\alpha(x)}, p(0,y)=i e^{i \beta(y)}. \]
To compute $\widehat{U}^{-1} \widehat{U}_x = L_-^{-1} \eta_+ L_- + L_-^{-1} L_{-_x}$ we let
\[L_- = \left(\begin{array}{cc}1 & \ell_{-1} \lambda^{-1} \\
 -\bar{\ell}_{-1}\lambda^{-1} & 1 \end{array}\right) \mod \lambda^{-2}.\]
This implies 
\[L_-^{-1} = \left(\begin{array}{cc}1 & -\ell_{-1} \lambda^{-1} \\
 \bar{\ell}_{-1}\lambda^{-1} & 1 \end{array}\right) \mod \lambda^{-2},\]
 and hence
 \[L_-^{-1} L_{-_x} = 0 \mod \lambda^{-1}.\]
 Thus
\begin{eqnarray*} 
\left(\begin{array}{cc} i r & \lambda q \\ -\lambda \bar{q} & -i r \end{array}\right) \doteq && \\
&\hspace{-1in} \left(\begin{array}{cc}1 & \ell_{-1} \lambda^{-1} \\
 -\bar{\ell}_{-1}\lambda^{-1} & 1 \end{array}\right) 
 \left(\begin{array}{cc}
0 & e^{-i \alpha(x)} \\
e^{i \alpha(x)} & 0 
\end{array} \right)
 \left(\begin{array}{cc}1 & -\ell_{-1} \lambda^{-1} \\
 \bar{\ell}_{-1}\lambda^{-1} & 1 \end{array}\right). 
 \end{eqnarray*}
This yields $q(x,y)=i e^{-i\alpha(x)}$ and $r(x,y)=|\ell_{-1}(x,y) e^{i \alpha(x)}|^2$ and shows that both $q$ and $r$ are $C^0$ in $x$ and $C^1$ in $y$.  Similarly, $\widehat{U}^{-1}\widehat{U}_y = L_+^{-1} \eta_- L_+ + L_+^{-1} L_{+_y}$ implies $p(x,y) = i e^{i \beta(y)} \ell_{+_0}(x,y)$ where $|\ell_{+_0}(x,y)|=1$ and that $p$ is $C^1$ in $x$ and $C^0$ in $y$.    Hence altogether we obtain (compare \cite{DIS}):
\begin{thm} \label{omegas}
Under the assumptions of this section we obtain
\begin{eqnarray*}
\widehat{U}^{-1} \widehat{U}_y &=& -\frac{i}{2} \lambda^{-1}
\left(\begin{array}{cc}0 & e^{i\widehat{\phi}(x,y)} \\ e^{-i\widehat{\phi}(x,y)} & 0  \end{array}\right) =: \widehat{\omega}_2, \\
\widehat{U}^{-1} \widehat{U}_x &=& \frac{i}{2}
\left(\begin{array}{cc}r & \lambda e^{-i\alpha(x)} \\ 
\lambda e^{i\alpha(x)} & -r  \end{array}\right) =: \widehat{\omega}_1 ,
\end{eqnarray*}
with real functions $\alpha,r$ ($C^0$ in $x$, $C^1$ in $y$) and a real function  $\widehat{\phi}$ ($C^0$ in $y$, $C^1$ in $x$) and $r(x,0)=0,\widehat{\phi}(0,y)=\beta(y)$.  Here we let $p(x,y) =  i e^{i \beta(y)} \ell_{+_0}(x,y) = i e^{i \widehat{\phi}(x,y)}$ for a real valued function $\widehat{\phi}(x,y)$.
\end{thm}
\begin{cor}\label{coromega}
Writing $\widehat{\omega}=\widehat{U}^{-1} d\widehat{U} = \widehat{\omega}_1 dx + \widehat{\omega}_2 dy$ we have that 
$\widehat{\omega}_1$ is $C^0$ in $x$ and $C^1$ in $y$; and $\widehat{\omega}_2$ is $C^0$ in $y$ and $C^1$ in $x$.  In particular it follows that $\widehat{U}_{xy}$ and $\widehat{U}_{yx}$ exist.
\end{cor}

\subsection{The Integrability Condition for {\huge $\widehat{\omega}$}}
Since we started from $C^0$-potentials we obtain that $\widehat{U}$ is $C^1$ in $(x,y)$.  However, Corollary \ref{coromega} shows that we actually can compute the zero curvature condition (ZCC) for
\[\widehat{U}_x = \widehat{U} \widehat{\omega}_1,\quad \widehat{U}_y = \widehat{U} \widehat{\omega}_2,\]
\begin{equation}
Z = \widehat{\omega}_{1_y}-\widehat{\omega}_{2_x}+[\widehat{\omega}_2,\widehat{\omega}_1].
\end{equation}
Although we know $\widehat{U}_{xy}$ and $\widehat{U}_{yx}$ exist, we don't know apriori if $\widehat{U}_{xy}=\widehat{U}_{yx}$, hence we proceed in a different way.  Clearly, $\widehat{U}_x=\widehat{U} \widehat{\omega}_1$ implies
\begin{equation} \label{Ufirst}
\widehat{U}(x,y)=\int_0^x \widehat{U}(t,y) \widehat{\omega}_1(t,y) dt +\widehat{U}(0,y). 
\end{equation}
This is because every term can be differentiated for $x$.  Hence the right side satisfies $W_x = W \widehat{\omega}_1$.  For fixed $y$ we evaluate the initial condition at $0$ and (\ref{Ufirst}) follows.  Since $\widehat{U}(x,y)$ and $\widehat{U}(0,y)$ are $C^1$ in $y$ we obtain by differentiation and (\ref{birkform2})
\begin{equation} \label{Uo1}
\widehat{U}(x,y) \widehat{\omega}_2(x,y) = \partial_y \Big(\int_{0}^x \widehat{U}(t,y) \widehat{\omega}_1(t,y) dt \Big) + \widehat{U}(0,y) \widehat{\omega}_2(0,y).
\end{equation}
We need to work out the remaining differentiation.  Since the integrand is $C^1$ in $y$, we can interchange integration and differentiation and obtain:
\[\partial_y \Big( \int_{0}^x \widehat{U}(t,y) \widehat{\omega}_1(t,y) dt \Big) =
\int_{0}^x [\widehat{U}(t,y)\widehat{\omega}_2(t,y)\widehat{\omega}_1(t,y) + \widehat{U}(t,y) \widehat{\omega}_{1_y}(t,y)] dt.\]
Thus (\ref{Uo1}) reads
\begin{eqnarray*}
\widehat{U} \widehat{\omega}_2 &= &\partial_y \Big( \int_{0}^x \widehat{U} \widehat{\omega}_1 \Big) + \widehat{U}(0,y) \widehat{\omega}_2(0,y) \\ &=&  \Big( \int_{0}^x \underbrace{\widehat{U}_y}_{\widehat{U} \widehat{\omega}_2} \widehat{\omega}_1 + \widehat{U} \widehat{\omega}_{1_y} \Big) + \widehat{U}(0,y) \widehat{\omega}_2(0,y).
\end{eqnarray*}
From this expression we see that each of the terms occurring now are $C^1$ in $x$.  After differentiation for $x$ we obtain
\[  \underbrace{\widehat{U}_x}_{\widehat{U} \widehat{\omega}_1} \widehat{\omega}_2 + \widehat{U} \widehat{\omega}_{2_x} = \widehat{U}\widehat{\omega}_2 \widehat{\omega}_1 + \widehat{U} \widehat{\omega}_{1_y}. \]
We have proved the following 
\begin{thm}
\[\widehat{U}_{xy} = \widehat{U}_{yx}.\]
In other words
\begin{equation} \label{zcc}
\widehat{\omega}_{1_y}-\widehat{\omega}_{2_x}+[\widehat{\omega}_2,\widehat{\omega}_1]=0.
\end{equation}
\end{thm}
\noindent Equation (\ref{zcc}) is the ZCC.  Next we want evaluate this equation by using the form of $\widehat{\omega}_1, \widehat{\omega}_2$ given in Theorem \ref{omegas}.  Equation (\ref{zcc}) thus reads

\[ \frac{i}{2}
\left(\begin{array}{cc} r_y & 0 \\ 0 & - r_y \end{array}\right)
+ 
\frac{i}{2}\lambda^{-1}
\left(\begin{array}{cc} 0 &  i \widehat{\phi}_x e^{i \widehat{\phi}(x,y)} \\ -i \widehat{\phi}_x e^{-i \widehat{\phi}(x,y)} & 0 \end{array}\right)
\]
\[
+ \frac{1}{4}
\left[\left(\begin{array}{cc} 0 &  \lambda^{-1} e^{i \widehat{\phi}(x,y)} \\ \lambda^{-1} e^{-i \widehat{\phi}(x,y)} & 0 \end{array}\right),
\left(\begin{array}{cc} r & \lambda e^{-i \alpha(x)} \\ \lambda e^{i \alpha(x)} & - r \end{array}\right)\right]
=0\]
By comparing the $(1,1)$-entries and the $(1,2)$-entries, we obtain respectively
\[ \left. \begin{array}{c}
r_y + \frac{1}{2i}(e^{i(\widehat{\phi}+\alpha)} - e^{-i(\widehat{\phi}+\alpha)}) =0, \\
r = -\widehat{\phi}_x. \end{array} \right.  \]
These two equations show that $\widehat{\phi}_x$ is differentiable for $y$ and we obtain
\begin{thm} The integrability condition is 
\[\widehat{\phi}_{xy} = \sin(\widehat{\phi} + \alpha).\] 
\end{thm}
\noindent Recall $\alpha=\alpha(x)$ is $C^0$, $\widehat{\phi}_{xy}$ exists, but $\widehat{\phi}_{y}$ may not.  Continuing to compute we have
\[\widehat{\phi}_x(x,y) = \int_{0}^y \sin(\widehat{\phi} + \alpha) dt + \widehat{\phi}_x(x,0), \]
and hence
\begin{equation} \label{phiintegral}
 \widehat{\phi}(x,y) =\int_0^x \Big( \int_0^y \sin(\widehat{\phi} + \alpha) dt \Big) ds 
+ \widehat{\phi}(x,0) + \widehat{\phi}(0,y) - \widehat{\phi}(0,0).
\end{equation}
Recall $r(x,0)=0$ and $r(x,y)=-\widehat{\phi}_x(x,y)$.  Thus $\widehat{\phi}_x(x,0)=0$ and therefore $\widehat{\phi}(x,0)=\widehat{\phi}(0,0)$.  Also recall $\beta(y)=\widehat{\phi}(0,y)$.  Thus equation (\ref{phiintegral}) reduces to
\[\widehat{\phi}(x,y)-\beta(y) = \int_0^x \Big( \int_0^y \sin(\widehat{\phi} + \alpha) dt \Big) ds.\] 
Setting $\widetilde{\phi}(x,y) = \widehat{\phi}(x,y)+ \alpha(x)$  yields
\[\widetilde{\phi}(x,y) -\beta(y) -\alpha(x) = \int_0^x \Big( \int_0^y \sin(\widetilde{\phi}) dt\Big) ds = \int_0^y \Big( \int_0^x \sin(\widetilde{\phi}) ds\Big) dt. \]
If we now let \[\check{\phi}=\widetilde{\phi}-\alpha -\beta, \]
then $\check{\phi}_{xy}$ exists, and $\check{\phi}_{yx}$ exists, and we have

\begin{thm} \label{weak} With $C^0$ input, $\eta=(\eta_-,\eta_+)$, our algorithm produces $\check{\phi}:D \stackrel{C^{1M}}{\longrightarrow} \mathbb{R}$ and $\widetilde{\phi}:D \stackrel{C^0}{\longrightarrow} \mathbb{R}$ with \[\check{\phi}_{xy} = \sin(\check{\phi}+\alpha+\beta) = \check{\phi}_{xy}\]
and
 \[\widetilde{\phi}_{xy} \stackrel{\mbox{{\normalsize weak}}}{=} \sin{\widetilde{\phi}}
\stackrel{\mbox{{\normalsize weak}}}{=} \widetilde{\phi}_{yx}\]
in the sense of distributions.
\end{thm}

\subsection{The $C^1$ Potential Case}
We would also like to compare our present results to \cite{DIS}.  Assume for this comparison that $\eta_-,\eta_+$ are actually $C^1$.  Then the first (unnumbered) equation in \cite{DIS} Section 2.5 is:
\[U_-^{-1} U_+ = V_+V_{-_0}^{-1} \cdot V_{-_0} T_-^{-1} V_{-_0}^{-1}.\]
Since $T_-$ and $V_{-_0} T_- V_{-_0}^{-1}$ start with $I$, we obtain
\[L_+ = V_+V_{-_0}^{-1},\;\; L_-=V_{-_0} T_-V_{-_0}^{-1}.\]
Hence the frame $U$ of \cite{DIS} relates to the frame $\widehat{U}$ of this paper as $\widehat{U} = U V_{-_0}^{-1}$.  Note that $\widehat{U}$ is $C^2$ in $(x,y)$ while $U$ is only $C^1$ in $(x,y)$.
In view of Theorem \ref{omegas} this corresponds precisely to the transition from $\phi$ in \cite{DIS} to our $\widehat{\phi} = \phi - \alpha$.  $\phi$ is the angle between the asymptotic lines.  (In Theorem \ref{weak} above we denote this angle by $\widetilde{\phi}$.)

\subsection{Frame to Immersion}
We have seen in the previous sections that $\widehat{U}(x,y,\lambda)$ is $C^1$ in $(x,y)$ and holomorphic in $\lambda \; \in \mathbb{C}^*$ (restricted to $\lambda>0$ for geometric purposes).  We set, using Sym's formula,
\begin{equation}  \label{f} 
f(x,y,\lambda) := \widehat{U}(x,y,\lambda)_t \widehat{U}(x,y,\lambda)^{-1}
\end{equation}
where $\lambda = e^t$ and

\begin{equation}  \label{n}
N(x,y,\lambda) := \widehat{U}(x,y,\lambda)  \left( \begin{array}{cc}  \frac{i}{2} & 0 \\  0 & -\frac{i}{2} \end{array} \right) \widehat{U}(x,y,\lambda)^{-1}.
\end{equation}

\begin{prop} \label{pot2surf} Let $f:D \longrightarrow \mathbb{R}^3$  be a map derived using a $C^0$ input potential $\eta$ and defined by Equation (\ref{f}).  Then
\mbox{}
\begin{enumerate}[label={\rm{\alph*)}},ref={\alph*)}]
\item  \label{pot2surfa}  We can interchange the differentiation for $t$ with the differentiation for $x$ and $y$ in equation (\ref{f}).  It follows that $f$ is $C^1$ in $(x,y)$ and holomorphic in $\lambda$.
\item  \label{pot2surfb} \mbox{} \vspace{-.3in}  \begin{eqnarray}
f_x &=& \widehat{U} \widehat{\omega}_{1_t} \widehat{U}^{-1} \\
&=& \frac{i}{2} \lambda \widehat{U}  \left( \begin{array}{cc}  0 & e^{-i \alpha(x)} \\  e^{i \alpha(x)} & 0 \end{array} \right) \widehat{U}^{-1}.
\end{eqnarray} 
\item  \label{pot2surfc} \mbox{} \vspace{-.3in} \begin{eqnarray}
f_y &=& \widehat{U} \widehat{\omega}_{2_t} \widehat{U}^{-1} \\
&=& \frac{i}{2} \lambda^{-1} \widehat{U}  \left( \begin{array}{cc}  0 & e^{i \widehat{\phi}(x,y)} \\  e^{-i \widehat{\phi}(x,y)} & 0 \end{array} \right) \widehat{U}^{-1}.
\end{eqnarray}
\item  \label{pot2surfd}  \mbox{} \vspace{-.3in}\begin{eqnarray}
f_x \times f_y &=& \sin(\widehat{\phi} + \alpha) N,\\
\Vert f_x \times f_y \Vert &=& |\sin(\widehat{\phi} + \alpha)|.
\end{eqnarray}
\end{enumerate}
\end{prop}
\begin{proof}
\ref{pot2surfa} The claim is true for $U_-$ and $U+$ since they are converging power series in $\lambda$.  The claim follows for $\widehat{U}$ by using equation (\ref{mainUsplit}), and hence for $f$.\\
\ref{pot2surfb} \&  \ref{pot2surfc} Using  \ref{pot2surfa} we obtain:
\[ (\widehat{U}_t \widehat{U}^{-1})_x = \widehat{U}_{xt} \widehat{U}^{-1} -\widehat{U}_t \widehat{U}^{-1} \widehat{U} \widehat{U}^{-1} = \widehat{U} \widehat{\omega}_{1_t} \widehat{U}^{-1}.\]
Therefore \ref{pot2surfb} (and similarly \ref{pot2surfc}) follow from Theorem \ref{omegas}.\\
\ref{pot2surfd} Since the cross product in $\mathbb{R}^3$ corresponds to the commutator in $su(2)$, we need to compute:
\begin{equation}\begin{split}
[\frac{i}{2} \lambda & \left(\begin{array}{cc} 0 & e^{-i \alpha(x)} \\ e^{i \alpha(x)} & 0 \end{array} \right),
\frac{i}{2} \lambda^{-1} \left(\begin{array}{cc} 0 & e^{i \widehat{\phi}(x,y)} \\ e^{-i \widehat{\phi}(x,y)} & 0 \end{array} \right) ]\\
&= -\frac{1}{4} \left(\begin{array}{cc} e^{-i(\widehat{\phi}+\alpha)} -e^{i(\widehat{\phi}+\alpha)} & 0 \\
0 & e^{i(\widehat{\phi}+\alpha)} -e^{-i(\widehat{\phi}+\alpha)}  \end{array} \right) \\
&= \frac{1}{2} \sin(\widehat{\phi}+\alpha) \left(\begin{array}{cc} i & 0 \\ 0 & -i \end{array} \right). \qedhere
\end{split} 
\end{equation}
\end{proof}

\noindent Recall from the introduction, that a parametrization of $f$  is called asymptotic if $f_x$ and $f_y$ always point in asymptotic directions (that is $\ell=0=n$) and it is called a Chebyshev parametrization if $<f_x,f_x>=1$ and $<f_y,f_y>=1$ at all points of $D$. \\
\noindent We are now able to prove the first half of our main theorem.
\begin{thm} \label{mixf} If the input potential $\eta$ is $C^0$, then 
\mbox{}
\begin{enumerate}[label={\rm{\alph*)}},ref={\alph*)}]
\item  \label{mixfa}The map $f$ of (\ref{f}) is $C^{1M}$, $<f_x,f_x> = \lambda^2$, $<f_y,f_y>=\lambda^{-2}$ and $<f_x,f_y> = \cos(\widehat{\phi} + \alpha)$.  In particular $f$ is weakly regular for all $\lambda$ and Chebychev if $\lambda=1$.
\item  \label{mixfb}  We have a ``generalized second fundamental form": $\ell=0, m=\sin(\widehat{\phi}+\alpha),n=0$.  In particular $f$ is asymptotic with $K=-1$ for all $\lambda$.
\item  \label{mixfc}$f_x = N \times N_x$, $f_y = -N \times N_y$.  In particular $N$ is weakly regular for all $\lambda$. 
\item  \label{mixfd}$N$ is $C^{1M}$ (since $\widehat{U}$ is) and Lorentz harmonic (see Definition \ref{defnN} below) with $N_{yx}=N_{xy} = \cos(\widehat{\phi} + \alpha) N$.  
\end{enumerate}
\end{thm}
\begin{rem} \label{twoN}
By Proposition \ref{pot2surf} \ref{pot2surfd} our $N$ differs by at most a sign from $N_{standard} = \frac{f_x \times f_y}{\Vert f_x \times f_y \Vert}$ whenever $\sin(\widehat{\phi}+\alpha) \neq 0$.  Thus, the standard second fundamental from might have $m=-\sin(\widehat{\phi}- \alpha)$, but we would still have $\ell=n=0$ and $K=-1$.
\end{rem}
\begin{proof}
\ref{mixfa} A direct calculation gives
\[\widehat{U}^{-1} (f_{xy}-f_{yx}) \widehat{U} = (\widehat{\omega}_{1_y} -\widehat{\omega}_{2_x} + [\widehat{\omega}_2,\widehat{\omega}_1])_t =0.\]
Thus $f$ is $C^{1M}$.  The calculation of the first fundamental form is straightforward.\\
\ref{mixfb} 
Recall $<A,B> = -2 \trace AB$ and again let  $\hat{e}_3=\left( \begin{array}{cc}  \frac{i}{2} & 0 \\  0 & -\frac{i}{2} \end{array} \right)$.
\begin{eqnarray*}
\ell &=& -<f_x,N_x> = 
-2 \trace(
\frac{i}{2} \lambda 
\left(\begin{array}{cc}0&e^{-i \alpha} \\ e^{i \alpha} & 0 \end{array}\right) [\hat{e}_3,\widehat{\omega}_1]), \\
&=& i \lambda \trace(\left(\begin{array}{cc}0&e^{-i \alpha} \\ e^{i \alpha} & 0 \end{array}\right)
 (\frac{\lambda}{4})\left(\begin{array}{cc}0&2e^{-i \alpha} \\ -2e^{i \alpha} & 0 \end{array}\right)  ) = 0. \\
m &=& -<f_x, N_y>, \\
&=& -2 \trace( \frac{i}{2} \lambda \left(\begin{array}{cc}0&e^{-i \alpha} \\ e^{i \alpha} & 0 \end{array}\right)  [\hat{e}_3, \widehat{\omega}_2]) ,\\
&=&  -i \lambda \trace (\left(\begin{array}{cc}0&e^{-i \alpha} \\ e^{i \alpha} & 0 \end{array}\right) (\frac{\lambda^{-1}}{4})
\left(\begin{array}{cc}0&2 e^{i \widehat{\phi}} \\ -2e^{-i \widehat{\phi}} & 0 \end{array}\right) ) \\
&=& -\frac{i}{2} (-e^{-i(\widehat{\phi}+\alpha)} + e^{i(\widehat{\phi} + \alpha)}),\\
&=& \sin(\widehat{\phi}+\alpha).
\end{eqnarray*}
Similarly $m=-<f_y,N_x>$, $n=0$ and $K= -1$ follows.\\
\ref{mixfc} and \ref{mixfd} Similar direct calculations.
\end{proof}

\section{From $C^{1M}$ PS-Fronts to $C^0$ Potentials}\label{converse}
\subsection{Weak (Lorentz) Harmoniticity and PS-Fronts}

\begin{defn}  \label{defnN} Let $N:D \longrightarrow \mathbb{S}^2$ be a $C^{1M}$ weakly regular immersion.  Then $N$ is called weakly (Lorentz) harmonic if there exists a function $h:D \longrightarrow \mathbb{R}$ such that $N_{xy}(p) = N_{yx}(p) = h(p) N(p)$ for all $p \in D$. 
\end{defn}

\begin{thm} \label{psthm} Let $N:D \longrightarrow \mathbb{S}^2$ be weakly harmonic.  Then there exists a weakly-regular $C^{1M}$ $f:D  \longrightarrow \mathbb{R}^3$ such that $f_x=N \times N_x$, $f_y= -N \times N_y$.  Moreover, any such $f$ is immersed at $p \in D$ if and only if $N$ is immersed at $p$, and $f$ is weakly-regular at $p$ if and only if $N$ is weakly-regular at $p$.
\end{thm}
\begin{proof} Note that the system
\begin{eqnarray*}
f_x & = & N \times N_x \\
f_y & = & - N \times N_y
\end{eqnarray*}
is solvable: 
\[f_{xy} = N_y \times N_x + N \times N_{xy} = N_y \times N_x \]
and 
\[f_{yx} = -N_x \times N_y - N \times N_{yx} = N_y \times N_x. \]
So we can define a map $f:D \longrightarrow \mathbb{R}^3$ which is $C^{1M}$.  If $N$ is weakly-regular at $p$ then $N_x(p) \neq \vec{0}, N_y(p) \neq \vec{0}, N \perp N_x, N \perp N_y$.  Then $f_x(p) \neq \vec{0}$ and $f_y(p) \neq \vec{0}$, so $f$ is weakly-regular at $p$.  If furthermore $N$ is regular at $p$, then $N_x \not\parallel N_y$ and hence $N || N_x \times N_y$.  We claim this implies $N \times N_x \not\parallel N \times N_y$.  If these nonzero vectors were parallel, then it would follow that $(N \times N_y) \times N_x$ and $(N \times N_y) \perp N_y$.  Hence $N_x || N_y$, a contradiction.  So $f_x \not\parallel f_y$, so $f$ is regular at $p$.
Conversely if $N$ is not regular at $p$ then $N_x(p) || N_y(p)$ which implies $f_x(p) || f_y(p)$ so $f$ is not regular at $p$.  If furthermore $N$ is not regular and not even weakly-regular at $p$ then either $N_x(p)=0$ or $N_y(p) =0$, which implies $f_x(p)=0$ or $f_y(0)=0$, so $f$ is not weakly-regular at $p$.
\end{proof}


\noindent Theorem \ref{psthm} allows us to make the following definition.  Compare with the definition of a front in \cite{SUY}.
\begin{defn}
Let $f:D \stackrel{C^{1M}}{\longrightarrow} \mathbb{R}^3$ be a weakly-regular immersion.  Then $f$ is called a weakly-regular $C^{1M}$ ps-front if there exists a weakly harmonic $N:D  \stackrel{C^{1M}}{\longrightarrow} \mathbb{S}^2$  such that $f_x=N \times N_x$, $f_y= -N \times N_y$. 
\end{defn}

\noindent Recall again from the introduction, that a parametrization of $f$  is called asymptotic if $f_x$ and $f_y$ always point in asymptotic directions (that is $\ell=0=n$) and it is called a Chebyshev parametrization if $<f_x,f_x>=1$ and $<f_y,f_y>=1$ at all points of $D$.
\noindent  We will show that any weakly-regular $C^{1M}$ ps-front $f$ is asymptotic (Equation (\ref{asymp})) and 
that by a $C^{1M}$ reparametrization any weakly-regular $C^{1M}$ ps-front can be made Chebyshev (Equation (\ref{cheby})), so without loss of generality we will assume this.  Finally, in this subsection, we will show that if $f$ is regular at $p$ then $K(p)=-1$.  In the previous section we proved Theorem \ref{mixf} which states that if $\eta$ is $C^0$, and $f_\eta$ is the surface derived from $\eta$ by formula (\ref{f}), then $f_\eta$ is a weakly regular $C^{1M}$ asymptotic Chebyshev ps-front.  The goal of this section is to prove the converse.
\begin{thm}\label{maintwo}
If $f$ is a weakly regular $C^{1M}$ asymptotic Chebyshev ps-front, then there exist a $C^0$ potential $\eta$ such that $f=f_\eta$.
\end{thm}
\noindent The proof of Theorem \ref{maintwo} will consist of the rest of this section.
\noindent First we prove that $f$ can be reparametrized by a $C^{1M}$-diffeomorphism to Chebyshev coordinates. We have that $<f_x,f_x>$ is differentiable in $y$ and 
\[\partial_y<f_x,f_x> = 2<f_{xy},f_x>=2<N_y \times N_x, N \times N_x> = 0, \]
since $N_y,N_x \perp N$ implies $N_y \times N_x \parallel N$, but $N \times N_x \perp N$.  Hence $<f_x,f_x> = E(x)$.  Similarly, $<f_y,f_y> = G(y)$.  Here $E$ and $G$ are continuous functions, so
\begin{equation} \label{cheby}
s(x) = \int_0^x \sqrt{<f_x(t,y),f_x(t,y)>}\; dt \; = \int_0^x \sqrt{E(t)}\; dt 
\end{equation}
is an invertible $C^1$-function of $x$.
So we can change coordinates, $(x,y) \rightarrow (x(s),y)$, and obtain without loss of generality $<f_x,f_x>=E \equiv 1$. 
Similarly we can change the y-coordinate and obtain $<f_y,f_y> =G \equiv 1.$  Let $F=<f_x,f_y>$ in these new coordinates.
Furthermore, by definition and since $<N_x, N>=0$ and $<N,N_y>=0$, it follows that in these coordinates we have
\begin{equation} 
<N_x,N_x>=1, <N_y,N_y> =1.
\end{equation}
We define 
\[f_x^\perp = N \times f_x \;\mbox{and}\; f_y^\perp = f_y \times N.\]
Note that $N_x = -f_x^\perp$ and $N_y=-f_y^\perp$. Moreover it is easy to see that 
${f_x,f_x^\perp, N}$ and ${f_y^\perp,f_y,N}$ are positively oriented orthonormal frames with \[\det(f_x,f_x^\perp,N)=\det(f_y^\perp,f_y,N)=1.\]
This allows us to unambiguously define the oriented angle $\omega$ from $f_x$ to $f_y$ in the oriented plane spanned by $f_x$ and $f_x^\perp$ as follows.
\begin{defn} \label{defomega} $\omega:D \longrightarrow \mathbb{R}$ is the unique function such that
\begin{enumerate}
\item $\omega(x,y) \!\!\!\mod \!2\pi = \angle( \f_x(x,y),\f_y(x,y)),$
\item $0 \leq \omega(0,0) < 2\pi,$
\item $\omega$ is $C^0$.
\end{enumerate}
\end{defn}
\noindent We immediately have that $F = <f_x,f_y> = \cos \omega$.  Moreover we can now clarify the distinction between our ``two normals" (both in asymptotic coordinates).  Notationally we let 
\[N_{front} :=N,\]
which is defined on all of $D$ and
\[N_{standard} := \frac{\f_x \times \f_y}{\Vert \f_x \times \f_y \Vert} \]
whenever $\Vert \f_x \times \f_y \Vert = |\sin \omega| \neq 0$ (and undefined otherwise).
We have $1=\det(f_x,f_x^\perp,N) = \det(f_x, \frac{f_y-\cos \omega f_x}{\sin \omega},N)$.  So $\sin \omega = \det(f_x,f_y,N)$ and $f_x \times f_y = \sin \omega \; N$.  In summary
\[N_{standard} = \frac{\sin \omega}{|\sin \omega|} N_{front}=\frac{\sin \omega}{|\sin \omega|} N,\]
whenever both sides are defined.  It is interesting that Toda's algorithm combined with Sym's formula selects $N_{front}$ not $N_{standard}$.

\noindent We can compute the second fundamental form (see Remark \ref{twoN}).
\begin{equation}\label{asymp}
\ell=-<f_x,N_x> = 0
\end{equation}{asymp}
since $f_x \perp N_x$ by definition of $f$.  Likewise $n=0$.
\begin{eqnarray*}
m &= &-<f_x,N_y> \\
& = &  -\partial_y <f_x,N> + <f_{xy},N>\\
&= & 0 + <f_{yx},N> \\
&=& \partial_x <f_y,N> - <f_y,N_x> \\
&=& -<f_y,N_x>
\end{eqnarray*}
So, $m=-<f_x,N_y>=-<f_y,N_x> = <f_{xy},N> = <f_{yx},N>$.

\noindent Now we compute $f_x \times f_y$.  We use the Jacobi identity $A \times (B \times C) = -C \times (A \times B) - B \times (C \times A)$ and obtain
\begin{eqnarray*}
f_x \times f_y &=& (N \times N_x) \times (-N \times N_y) \\
&=& N_y \times ((N \times N_x) \times N) + N \times (N_y \times (N \times N_x)) \\
&=& (N \times (N \times N_x)) \times N_y + *  \\
&=& (N <N,N_x> - N_x <N,N>) \times N_y + * \\
&=& -N_x \times N_y +  * 
\end{eqnarray*}
so $f_x \times f_y = - N_x \times N_y \; + *$, where $* =  N \times (N_y \times (N \times N_x))$ is 
perpendicular to $N$ and $-N_x \times N_y$ is parallel to $N$.
But $f_x \times f_y \; ||\; N$, so
\[f_x \times f_y = -N_x \times N_y \; \mbox{and} \;  N_x \times N_y =  -\sin \omega \; N \; \mbox{follows}.\]
From this we can compute $m$ in terms of $\omega$.
\begin{eqnarray*}
m &=&-<f_x,N_y> \\
&=&-\det(N,N_x,N_y) \\
&=&-\det(N_x,N_y,N) \\
&=&-<N_x \times N_y, N> \\
&=& \sin \omega.
\end{eqnarray*}
\noindent Let $N_x = \alpha f_x + \beta f_y$ and $N_y = \gamma f_x + \delta f_y$.  So, taking inner products against $f_x$ and $f_y$, we have
\[\left( \begin{array}{cc} 0 & -m \\ -m & 0 \end{array} \right) =
\left(\begin{array}{cc} 1 & F \\ F & 1  \end{array} \right)
\left(\begin{array}{cc} \alpha & \gamma \\ \beta & \delta  \end{array} \right). \]
Computing $\left(\begin{array}{cc} 1 & F \\ F & 1  \end{array} \right)^{-1}\left( \begin{array}{cc} 0 & -m \\ -m & 0 \end{array} \right)$ and using $m=\sin \omega$ and $F=\cos \omega$
 gives (at regular points) $\alpha = \delta = \frac{mF}{1-F^2} =
 \frac{\cos \omega}{\sin \omega}$ and $\beta = \gamma = 
 \frac{-m}{1-F^2}=-\frac{1}{\sin \omega}$.  In particular
 \[ \det \left(\begin{array}{cc} \alpha & \gamma \\ 
 \beta & \delta \end{array} \right)  = \alpha^2 - \beta^2 = -1.\]
 Hence $K$ $=-1$ (at regular points).
 
\begin{rem}
One can also determine $K$ by the spherical image, since $N$ is $C^1$.
\end{rem}
\noindent Finally we compute $h$ in the definition of the harmonicity of $N$.  Recall $N_{xy}(p) = N_{yx}(p) = h(p) N(p)$.  First note $N \times f_x = N \times (N \times N_x) = -N_x$.  This implies
\[ -N_{xy} = N_y \times f_x + N \times f_{xy}. \]
The second term is $0$ since $f_{xy} = N_y \times N_x$ is parallel to $N$.
So (at regular points)

\begin{eqnarray*}
N_{xy} &=& -(\gamma f_x + \delta f_y) \times f_x \\
&=&  \delta f_x \times f_y \\
&=&  \frac{\cos \omega}{\sin \omega} \sin \omega \; N \\
&=& \cos \omega \; N.
\end{eqnarray*}

\begin{rem}
We show that  $N_{xy} = 0$ on some open set contradicts weak-regularity.  
$N_{xy} =  0$ implies $\cos(\omega)=0$ which without loss of generality 
implies $\omega \equiv \pi/2$.  Moreover, it follows that $f_x = N_y$ and $f_y=N_x$ .  In particular we 
have $N_{xy}=0$ implies $N_x(x,y) = f_y(x,y)=f_y(x)$.  Hence  $f=f_y(x)*y+C(x)$.  Similarly $N_{yx}=0$ implies
$N_y(x,y) = f_x(x,y)=f_x(y)$.  Hence  $f=f_x(y)*x + D(y)$.  So $f$ is linear in both $x$ and $y$.  Say $f=Ax+By+C$ for 
vectors $A,B,C$ , then $A \neq 0$ and $B \neq 0$ (since $f_x \neq 0$ and $f_y \neq 0$ by weak-regularity).  So $f$ is a
plane and $N$ is constant.  But then $N_x=f_y=0$ (and  $N_y=f_x=0$), which is a contradiction.  Thus $N_{xy}=0$ on some open set cannot
occur.  
\end{rem}
\begin{rem}
Note that $c(x):=f(x,0)$ is a curve parametrized by arclength with tangent $\vec{t} = f_x(x,0)$, normal (assuming the curvature of $c(x) \neq 0$) $\vec{n} = \pm f_x^\perp(x,0)$, and binormal $N=\vec{b}=\vec{t} \times \vec{n}$.  Then the Serret-Frenet formulas give $-dN(\vec{t}(x)) = -\vec{b}'(x) = \tau(x) \vec{n}(x)$ where $\tau(x)$ is the torsion of the asymptotic curve $c(x)$.  Since $N_x(x,0) = -f_x^\perp(x,0)$ we have $\tau(x) = \pm 1$.  Thus the Beltrami-Enneper theorem, Item \ref{BelEnn} in the Introduction, continues to hold for our $C^{1M}$-immersions.
\end{rem} 

\subsection{Frames}
We begin by proving that any $N$ which is $C^{1M}$ lifts to a $C^{1M}$ orthonormal frame $E$.
\begin{lem}
Let $N:D \longrightarrow S^2$ be $C^{1M}$.  Let $\pi: SU(2) \longrightarrow S^2 \approx SU(2)/U(1)$ be the standard fiber bundle over $S^2$ and $\widetilde{N}:N^*SU(2) \longrightarrow SU(2)$ its pullback bundle over $N$.  Then there exists $E:D \longrightarrow SU(2)$ which is $C^{1M}$ and satisfies $N = \pi \circ E$.
\end{lem}
\begin{proof}
Since $D=J \times J$ is contractible, the pullback bundle $N^*SU(2)$ over $D$ is trivial.  Let $\psi:D \times U(1) \longrightarrow N^*SU(2)$ be a $C^{1M}$ global trivialization of $N^*SU(2)$.  Let $s_0$ be any point in $U(1)$ and let $\iota$ be the corresponding inclusion map $\iota(x)=(x,s_0)$.  Then $E:D \longrightarrow SU(2)$ defined by $\widetilde{N} \circ \psi \circ \iota$ is $C^{1M}$ and satisfies $\pi \circ E = N$.
\end{proof}
\noindent Define $C^0$ coefficient matrices, $W$  and $V$ by 
\[E_y = E W \;(i.e.\; W=E^{-1}E_y), \]
\[E_x = E V \;(i.e.\; V=E^{-1}E_x). \]
Since, by definition, $E$ is $C^{1M}$, $E_{yx}$ exists and is continuous, which implies $W$ is $C^1$ in $x$ and $C^0$ in $y$.  Similarly  $E_{xy}$ exists and is continuous, which implies $V$ is $C^1$ in $y$ and $C^0$ in $x$.

\noindent The compatibility condition 
\[W_x - V_y = [W,V] \]
can be computed as usual.  But before evaluating this equation we normalize the frame.  Write
\[W = W_{\mathfrak{k}} + W_{\mathfrak{p}}\]
using the Cartan decomposition $\mathfrak{k} + \mathfrak{p} = \mathfrak{g}$ of $so(3)$ relative to the involution $X \mapsto AXA^{-1}$ where $A=diag(-1,-1,1)$.  Then solve 
\[K^{-1} K_y = -W_{\mathfrak{k}}\]
as a function of $y$ with parameter $x$.  This gives a $K$ which is $C^1$ in $y$.  Moreover, since $W$ and hence $W_{\mathfrak{k}}$ is $C^1$ in $x$ we have that $K$ is $C^1$ in $x$.  So $K$ is $C^1$ in $(x,y)$ and we may replace $E$ with $EK$.  In summary, without loss of generality we may assume $E$ is $C^1$ and $W_{\mathfrak{k}} = 0$.
\begin{rem}
Keeping this normalization one can only gauge $E$ by some diagonal block gauge $K_0(x)$ commuting with $A$ which is $C^1$ in $x$.   (The notation is meant to indicate that $K_0$ is independent of $y$.)
\end{rem}
\noindent We have 
\begin{equation} \label{Enormal}
E^{-1} E_y = \left(\begin{array}{cc} 0 &p \\ -\bar{p} & 0 \end{array}\right) = W,
\end{equation}
\begin{equation}
E^{-1} E_x = \left(\begin{array}{cc} i r & q \\ -\bar{q} & -i r \end{array}\right) = V.
\end{equation}
Now the compatibility condition $W_x-V_y=[W,V]$ yields
\[\left(\begin{array}{cc}0 & p_x \\ -\bar{p}_x & 0 \end{array}\right) -
\left(\begin{array}{cc}i r_y & q_y \\ -\bar{q}_y & -i r_y \end{array}\right) =
\left(\begin{array}{cc}-p\bar{q} + q\bar{p} & -2 i p r \\
-2\bar{p} i r & -\bar{p} q + \bar{q}p \end{array}\right), \]
and hence
\begin{eqnarray*}
i r_y &=& p \bar{q} - q \bar{p}, \\
p_x - q_y &=& -2i p r .
\end{eqnarray*}
To understand what this means we go back to the $3 \times 3$ - picture.
Let $\theta = \frac{\omega}{2}$, $\omega$ as in Definition (\ref{defomega}).  We begin by following \cite{DIS}.
\[e_1 = \frac{1}{2 \cos \theta}(f_x + f_y), e_2 = \frac{1}{2 \sin \theta}(f_y-f_x), e_3=N,\]
\[\widetilde{F} = (e_1,e_2,N),\]
Note: We are using $N$ here (not $N_{standard}$) and $\widetilde{F}$ is well-defined even when $\cos \theta=0$ or $\sin \theta =0$.  We have $e_1 = \cos \theta f_x + \sin \theta f_x^\perp$, $e_2 = \cos(\theta +\frac{\pi}{2}) f_x + \sin(\theta + \frac{\pi}{2}) f_x^\perp$, and  $\det \widetilde{F} = \det(e_1,e_2,N)=1$.  In \cite{DIS} the next step is
\begin{eqnarray*}
F &=& \widetilde{F} \left(\begin{array}{ccc} \cos \theta & \sin \theta & 0 \\
-\sin \theta & \cos \theta & 0 \\
0 & 0 & 1 \end{array}\right), \\
&=& (\cos \theta e_1 - \sin \theta e_2, \sin \theta e_1 + \cos \theta e_2, N), \\
&=& (\frac{1}{2}(f_x+f_y)-\frac{1}{2}(f_y-f_x), 
\frac{1}{2}\frac{\sin \theta}{\cos \theta}(f_x+f_y) + 
\frac{1}{2}\frac{\cos \theta}{\sin \theta}(f_y-f_x), N), \\
&=& (f_x, \frac{1}{2} \frac{\sin^2 \theta - \cos^2 \theta}{\cos \theta \sin \theta} f_x +
\frac{1}{2}  \frac{\sin^2 \theta + \cos^2 \theta}{\cos \theta \sin \theta} f_y, N), \\
&=& (f_x, -\frac{\cos 2\theta}{\sin 2\theta}f_x+\frac{1}{\sin 2\theta}f_y,N).
\end{eqnarray*}
Note again that $F$ is well-defined even when $\sin 2\theta = 0$, with $F=(f_x,f_x^\perp,N)$ and $\det(f_x,f_x^\perp,N)=1$.

\noindent We have introduced five different frames, all along our asymptotic coordinates.  Let's review their properties.
\begin{enumerate}
\item $(f_x, f_y, N)$.  This is the asymtotic line frame along the asymptotic coordinates.  It is $C^{1M}$, but is not a frame if $f_x$ and $f_y$ are  collinear.  It is positively oriented only half the time, and is rarely orthonormal.
\item $\widehat{F}= (e_1,e_2,N)$.  This is the curvature line frame along asymptotic coordinates.  It is apriori only $C^0$, but is always defined, positively oriented, and orthonormal.
\item $F=(f_x,f_x^\perp,N)$.  This is the frame used in \cite{DIS}.  It is apriori only $C^0$, but is always defined, positively oriented, and orthonormal.
\item $E = (E_1, E_2, N)$ (unnormalized).  This is the frame constructed from $N$ at the beginning of this section.  It is $C^{1M}$, always defined, positively oriented, and orthonormal.
\item $E = (E_1,E_2,N)$ (normalized).  This is constructed from the unnormalized $E$ by a gauge $K$ as above and it is the version used below.  It is only $C^1$, but still is always defined, positively oriented, and orthonormal.
\end{enumerate}
\subsection{Finding the Input Potential}
Because the nice $C^1$ normalized coordinate frame $E$ does not work, we rotate this frame in the tangent plane.
\[F=E R\]
where $E$ is $C^1$, but $F$ is perhaps only $C^0$ and
\[R=\left(\begin{array}{ccc} \cos \tau & \sin \tau &0 \\
-\sin \tau & \cos \tau & 0 \\
0 & 0 & 1 \end{array} \right).\]
So
\begin{eqnarray}
\widehat{E} = E & = & F R^{-1} \nonumber \\
& =  & (\cos \tau f_x-\frac{\cos 2\theta}{\sin 2\theta} \sin \tau f_x + 
 \frac{\sin \tau}{\sin 2\theta} f_y, \nonumber \\
 &&\hspace{1in} -\sin \tau f_x - \frac{\cos \tau \cos 2\theta}{\sin 2\theta}f_x +
\frac{\cos \tau}{\sin 2\theta} f_y,N), \nonumber \\
& =  & (\frac{\sin(2\theta-\tau)}{\sin 2\theta}  f_x + 
 \frac{\sin \tau}{\sin 2\theta} f_y, \nonumber \\
 &&\hspace{1in} -\frac{\cos(2\theta-\tau)}{\sin 2\theta}f_x +
\frac{\cos \tau}{\sin 2\theta} f_y,N). \label{E123}
\end{eqnarray}
We use the $\widehat{E}$ notation to emphasize the $3 \times 3$ setting.  In the $3 \times 3$ setting the normalization (\ref{Enormal}) translates into
\[\widehat{E}_y = \widehat{E} \left(\begin{array}{ccc} 0 & 0 & -q \\
0 & 0 & -p \\
q & p &0 \end{array} \right).\]
With this normalization we have
\[\widehat{E}^{-1}\widehat{E}_y = W = \left(\begin{array}{ccc} 0 & 0 & -q \\
0 & 0 & -p \\
q & p &0 \end{array} \right),\]
\[\widehat{E}^{-1}\widehat{E}_x = V =  \left(\begin{array}{ccc} 0 & c & -a \\
-c & 0 & -b \\
a & b &0 \end{array} \right).\]
Let $\check{e}_1=\left(\begin{array}{c} 1 \\ 0 \\ 0 \end{array}\right), 
\check{e}_2=\left(\begin{array}{c} 0 \\ 1 \\ 0 \end{array}\right), 
\check{e}_3=\left(\begin{array}{c} 0 \\ 0 \\ 1 \end{array}\right)$ and
$\widehat{E}_1 = \widehat{E} \check{e}_1$, $\widehat{E}_2 = \widehat{E} \check{e}_2$, $\widehat{E}_3 = \widehat{E} \check{e}_3$.
Then $\widehat{E}=(\widehat{E}_1,\widehat{E}_2,\widehat{E}_3=N)$.  On the one hand
\[<\widehat{E}_{y} \check{e}_1, N> = <\widehat{E}W\check{e}_1,N> = <\widehat{E}q\check{e}_3,N> = q<N,N> = q. \]
On the other hand
\[<\widehat{E}_{y} \check{e}_1, N> = <\widehat{E}_{1_y},N> = \partial_y \!\!<\widehat{E}_1,N> - <\widehat{E}_1,N_y> =-<\widehat{E}_1, \gamma f_x + \delta f_y>\!\!.\]
So
\begin{equation}
q = -<\widehat{E}_1, \gamma f_x + \delta f_y>.
\end{equation}
From (\ref{E123}) we know ($2\theta = \omega$)
\begin{eqnarray*}
\widehat{E}_1 & = & \frac{\sin(\omega-\tau)}{\sin \omega}  f_x + 
 \frac{\sin \tau}{\sin \omega} f_y,\\
 \widehat{E}_2 & = & -\frac{\cos(\omega-\tau)}{\sin \omega}  f_x + 
 \frac{\cos \tau}{\sin \omega} f_y.
\end{eqnarray*}
Recall $\gamma = -\frac{1}{\sin \omega}, \delta = \frac{\cos \omega}{\sin \omega}$.  So we have
\begin{eqnarray*}
-q &=& - \frac{1}{\sin \omega} (\frac{\sin(\omega-\tau)}{\sin \omega} 1 +
\frac{\sin \tau}{\sin \omega} \cos \omega) +
\frac{\cos \omega}{\sin \omega}
(\frac{\sin(\omega-\tau)}{\sin \omega}\cos \omega +
\frac{\sin \tau}{\sin \omega} 1) \\
&=& \frac{\sin(\omega-\tau)}{\sin^2 \omega}(-1 +\cos^2 \omega) 
= -\sin(\omega-\tau).
\end{eqnarray*}
Hence 
\[q = \sin(\omega - \tau).\]
Similarly we have 
\[p=-\cos(\omega-\tau),\;\; a=\sin \tau,\;\; b= \cos \tau.\]
To return to the $2 \times 2$-picture we recall from \cite{DIS}.
\begin{eqnarray*}
J: \check{e}_1 \rightarrow \frac{1}{2}\left(\begin{array}{cc}
0 & i \\ i & 0 \end{array}\right) & = & \hat{e}_1, \\
 \check{e}_2 \rightarrow \frac{1}{2}\left(\begin{array}{cc}
0 & -1 \\ 1 & 0 \end{array}\right) & = & \hat{e}_2, \\
 \check{e}_3 \rightarrow \frac{1}{2}\left(\begin{array}{cc}
i&0\\ 0&-i \end{array}\right) & = & \hat{e}_3.
\end{eqnarray*}
Then
\begin{eqnarray*}
\widehat{E}_1 &\rightarrow& S \hat{e}_1 S^{-1}, \\
\widehat{E}_2 &\rightarrow& S \hat{e}_2 S^{-1}, \\
\widehat{E}_3 &\rightarrow& S \hat{e}_3 S^{-1}. 
\end{eqnarray*}
For $X \in su(2)$:
\[J\widehat{E}J^{-1}(X) = SXS^{-1}.\]
Hence 
\[J \widehat{E}_x J^{-1}(X) = S[S^{-1}S_x,X] S^{-1}.\]
Which implies
\[J \widehat{E}^{-1} \widehat{E}_x J^{-1} (X) = [S^{-1}S_x,X].\]
Put $X=u \hat{e}_1 + v \hat{e}_2 + w \hat{e}_3$, and we have
\begin{eqnarray*}
J \widehat{E}^{-1}\widehat{E}_xJ^{-1}(X) &=& J(\widehat{E}^{-1}\widehat{E}_x) \left(\begin{array}{c}
u \\ v \\ w \end{array}\right) =
JV \left(\begin{array}{c}u \\ v \\w \end{array}\right) \\
&=&J \left(\begin{array}{c}cv-aw \\ -cu-bw \\ au+bv\end{array}\right) \\
&=& (cv-aw)\hat{e}_1-(cu+bv)\hat{e}_2+(au+bv)\hat{e}_3.
\end{eqnarray*}
Write $S^{-1}S_x =
 \left(\begin{array}{cc}i \epsilon & \rho \\ 
 -\bar{\rho} & -i \epsilon\end{array}\right)$ with $\epsilon(x,y) \in \mathbb{R}$, $\rho(x,y) \in \mathbb{C}$.
 Computing we have
 \begin{eqnarray*}
 [S^{-1}S_x, \hat{e}_1] & = & [ \left(\begin{array}{cc}i \epsilon & \rho \\ 
 -\bar{\rho} & -i \epsilon\end{array}\right), \frac{1}{2}\left(\begin{array}{cc}
0 & i \\ i & 0 \end{array}\right)] \\
&=& \frac{1}{2}\left(\begin{array}{cc}\rho i + i \bar{\rho} & -\epsilon - \epsilon \\
\epsilon + \epsilon & -\bar{\rho} i - \rho i \end{array}\right) =
\frac{1}{2}\left(\begin{array}{cc}i Re \rho & -2\epsilon \\
2\epsilon & -i Re \rho \end{array}\right) \\
&=& 2 Re \rho \hat{e}_3 + 2\epsilon \hat{e}_2 = -c\hat{e}_2+a\hat{e}_3.
 \end{eqnarray*}
 Which implies $c=-2\epsilon, a=2 Re \rho$.  Similarly $b=2 Im \rho$.  So
 \[S^{-1}S_x = \frac{1}{2}\left(\begin{array}{cc}-i c & a+ib \\
-(a-ib) & i c \end{array}\right) =  i \frac{i}{2}\left(\begin{array}{cc}c & e^{-i \tau} \\
e^{i \tau} & -c \end{array}\right).\]
Similarly if we write $S^{-1}S_y =
 \left(\begin{array}{cc}0 & \sigma \\ 
 -\bar{\sigma} &0\end{array}\right)$ with  $\sigma(x,y) \in \mathbb{C}$.  Then we compute that $q = 2 \, Re \,\sigma, p= 2\, Im \,\sigma$ and 
  \[S^{-1}S_y  =  \frac{i}{2}\left(\begin{array}{cc}0 & e^{i(\omega- \tau)} \\
e^{-i(\omega-\tau)} & 0 \end{array}\right).\]
Finally we evaluate the integrability condition
\[W_x-V_y = [W,V]\]
that is
\begin{eqnarray*}
\left(\begin{array}{cc} 0  & \sigma_x \\ -\bar{\sigma}_x & 0 \end{array}\right)  -
\left(\begin{array}{cc} i \epsilon_y & \rho_y \\ -\bar{\rho}_y & -i \epsilon_y \end{array}\right) & = &
[\left(\begin{array}{cc} 0  & \sigma \\ -\bar{\sigma} & 0 \end{array}\right), 
\left(\begin{array}{cc} i \epsilon & \rho \\ -\bar{\rho} & -i \epsilon \end{array}\right) ] ,
\end{eqnarray*}
which yields
\begin{eqnarray*}
\left(\begin{array}{cc} 0  & -\frac{(\omega-\tau)_x}{2} e^{i(\omega-\tau)} \\ 
\frac{(\omega-\tau)_x}{2} e^{-i(\omega-\tau)} & 0 \end{array}\right)  -
\left(\begin{array}{cc}-\frac{i}{2}c_y & -\frac{\tau_y}{2} e^{-i\tau} \\
 \frac{\tau_y} e^{i \tau} & \frac{i}{2} c_y \end{array}\right) & = & \\
 \mbox{$[$} \frac{i}{2}\left(\begin{array}{cc} 0  & e^{i(\omega-\tau)} \\ 
e^{-i(\omega-\tau)} & 0 \end{array}\right), 
\frac{i}{2}\left(\begin{array}{cc} -c & e^{-i\tau} \\
e^{i \tau} & c \end{array}\right) ],  & &
\end{eqnarray*}
and a straightforward calculation gives
\begin{eqnarray*}
c_y & = & \sin \omega, \\
 (\omega-\tau)_x +  \tau_y e^{-i \omega}&=& c.
\end{eqnarray*}
Since $(\omega-\tau)_x$ and $c$ are real, the expression $\tau_y e^{-i \omega}$ is real.  Hence $\tau_y=0$.  Thus $\tau = \tau(x)$ and also $(\omega-\tau)_x = c$ which implies $(\omega-\tau)_x$ is $C^1$ in $y$ and 
\[(\omega-\tau)_{xy} = \sin \omega.\]
Replacing $\sigma$ by $\lambda^{-1} \sigma$ and $\rho$ by $\lambda \rho$ yields now the same integrability condition.  We have 
\[S^{-1} dS = \lambda^{-1} A dy + B + \lambda D dx\]
is integrable for all $\lambda > 0$.  Thus we have an extended frame $S=S(x,y,\lambda)$ which we can factor as
\[S=S_-L_+=S_+L_-.\]
As in \cite{DIS} Section 2.4 this yields the desired $C^0$ potentials $\eta_- =S_-^{-1} S_{-_y}$ and $\eta_+ = S_+^{-1} S_{+_x}$.  This concludes the proof of Theorem \ref{maintwo}.

\section{Hartman-Wintner Theory}\label{hwthy}
\begin{thm} \label{hwthm} (Hartman-Wintner \cite{HW}) 
Let $f_{orig}:D_{(u,v)} \stackrel{C^2}{\longrightarrow} \mathbb{R}^3$ be a regular immersion with $K=-1$.  Then, locally, there is a unique (up to orientations) $C^1$ reparametrization $T:\breve{D}_{(x,y)} \longrightarrow \breve{D}_{(u,v)}$ of $f_{orig}$ by asymptotic Chebyshev coordinates.  Furthermore, $f_{asyche}=f_{orig} \circ T$ is $C^{1M}$.  
\end{thm}
\noindent $\breve{D}$ denotes an appropriate small local domain about $(x,y)$.
\begin{rem}
Note, in \cite{HW} it is shown that this theorem is false if the condition $K=-1$ is weakened to $K<0.$
\end{rem}
\begin{rem}
\noindent Theorem \ref{hwthm} together with the results of the previous section imply that we can produce, locally, via Toda's algorithm, all immersions which are $C^2$ with $K=-1$ after some change of coordinates.
\end{rem}
\noindent As discussed in Section \ref{kimm} one can check that for such a $C^{1M}$ regular immersion the Gauss curvature and second fundamental form are defined as usual.  The second fundamental form is still symmetric and one still has
\[K = \frac{\ell n - m^2}{EG-F^2}.\]
\noindent Without loss of generality we assume $N=N_{asyche} = \frac{f_x \times f_y}{\sin{\theta}}$ where $\theta$, the angle from $f_x$ to $f_y$, satisfies $0 < \theta < \pi$.  Hence $<f_x,f_y> = \cos \theta$.  The condition of being asymptotic Chebyshev with $K=-1$ implies $<f_x,N_x>=0$, $<f_y,N_y>=0$ and (without loss of generality) $f_{xy}= \sin{\theta}\; N$.  We can strengthen Theorem \ref{hwthm} as follows.
\begin{thm} \label{betterhw}
With the same assumptions as in Theorem \ref{hwthm} we have $N=N_{asyche}$ is $C^{1M}$, $N_{xy}=N_{yx}=\cos \theta \;N$, $f_x = N \times N_x$ and $f_y=-N \times N_y$.  
\end{thm}
\begin{proof}
First note $<f_x,N_y> = \partial_y <f_x,N> - <f_{xy},N> = -\sin{\theta}$.  Similarly $<f_y,N_x> = -\sin \theta$.  Now let $N \times N_x = af_x+bf_y$.  We then have
\begin{eqnarray*}
<N \times N_x, f_x> &=& <f_x \times N, N_x> \\
&=& \frac{1}{\sin \theta} <f_x \times (f_x \times f_y), N_x> \\
&=& \frac{1}{\sin \theta} <f_x \times (f_x \times (f_y - \cos \theta f_x)), N_x> \\
&=& \frac{1}{\sin \theta} <-(f_y - \cos \theta f_x), N_x> \\
&=& 1,
\end{eqnarray*}
and similarly $<N \times N_x,f_y> = \cos \theta$.  Taken together we have $1=a+b \cos \theta$ and $\cos \theta = a \cos \theta + b$, which gives $b=0, a=1$ and hence $f_x=N \times N_x$.  Similarly $f_y = -N \times N_y$.  This also yields immediately $<N_x, N_x> =1$ and $<N_y,N_y> =1$.  Finally, by Lagrange's identity, $<N_x,N_y> = - \cos \theta$.  Next we claim that $N_{xy}$ exists.  $N \perp N_x$ implies $N \times f_x = N \times (N \times N_x) = -N_x$.  From this we have that $N_x$ is differentiable in $y$ and moreover
\[(N_x)_y = -(N \times f_x)_y =-(N_y \times f_x + N \times f_{xy})=-N_y \times f_x.\]
This implies $N_{xy}$ is parallel to $N$.  Similarly $N_{yx}$ is parallel to $N$.  But ${<\!N_{xy},N\!>}$ $= <N_x,N>_y - <N_x,N_y> = \cos \theta$ and similarly $<N_{yx},N> = \cos \theta$ and we have $N_{xy}=N_{yx}= \cos \theta N$.
\end{proof}
\begin{rem}\label{diffNback}
Note that $f_{orig}$ is $C^2$ and $f_{asyche}$ is $C^{1M}$, while $N_{orig}$ is $C^1$ and $N_{asyche}$ is $C^{1M}$.  Thus the change of coordinates decreases the differentiability of $f$ and increases the differentiability of $N$.
\end{rem}
\begin{thm} \label{hartwin} (Continuing Theorem \ref{betterhw})
Conversely, let $f_{asyche}:D_{(x,y)} \stackrel{C^{1M}}{\longrightarrow} \mathbb{R}^3$ be a regular immersion in asymptotic Chebyshev coordinates.  Assume  $N=N_{asyche}:= \frac{f_x \times f_y}{\sin{\theta}}$ is $C^{1M}$, $N_{xy}=N_{yx}=\cos \theta N$, $f_x = N \times N_x$, and $f_y=-N \times N_y$. (Note this implies  $K=-1$.)   Then, locally, there is a $C^1$ reparametrization $f_{graph}=f_{asyche} \circ S$ of $f_{asyche}$, where $S:\breve{D}_{(u,v)} \longrightarrow \breve{D}_{(x,y)}$, with $f_{graph} \; C^2$.
\end{thm}
\noindent  Here again we are assuming $\theta$, the angle from $f_x$ to $f_y$, satisfies $0 < \theta < \pi$.  The subscript  ``graph" is used because the type of coordinates used are often called graph coordinates. $\breve{D}$ denotes an appropriate small local domain about $(u,v)$.  A similar result is in \cite{HW}, but assuming more differentiability.  Our proof is a modification of theirs.
\begin{proof}
Consider $f(x,y)$ in asymptotic coordinates.  In a neighborhood of any point there is a rigid motion $R:\mathbb{R}^3 \longrightarrow \mathbb{R}^3$ so that $f$ is a graph.  Then $u(x,y)=\pi_1 (R(f(x,y)))$ and $v(x,y)=\pi_2(R(f(x,y))$ are both $C^1$ in $x,y$.  Let $\widehat{f}(u,v) = (u,v,h(u,v))$ denote the function in these new coordinates.  So 
\[\widehat{N}(u,v) = \frac{(-h_u,-h_v,1)}{\sqrt{1+h_u^2+h_v^2}}\]
and
\[N(x,y) = \widehat{N}(u(x,y),v(x,y)) = \frac{(-h_u(x,y),-h_v(x,y),1)}{\sqrt{1+h_u^2(x,y)+h_v^2(x,y)}}\]
is $C^1$ in $x,y$ by assumption.  From this it follows that $h_u(x,y)$ and $h_v(x,y)$ are $C^1$ in $x,y$.  In other words $h_u(u(x,y),v(x,y))$ and $h_v(u(x,y),v(x,y))$ are $C^1$ in $x,y$.  Since above we noted that both $u(x,y)$ and $v(x,y)$ are $C^1$ in $x,y$, we conclude that $h_u(u,v)$ and $h_v(u,v)$ are $C^1$ in $u,v$.  Which gives $h$ is $C^2$ in $u,v$.
\end{proof}

\section{A $C^{1M}$ version of Hilbert's Theorem}\label{hilbert}
\subsection{No $C^2$ Exotic $\mathbb{R}^2$'s}
In order to prove a $C^{1M}$ version of Hilbert's Theorem, it is necessary to first give a global version of Theorem \ref{hartwin}.  For that we need the following easy Lemma which follows directly from a result of Whitney.  Although perhaps well-known to the experts, it does not appear to be in the literature.
Let ${\cal{S}}^k$ denote the standard maximal $C^k$-atlas on $\mathbb{R}^2$.
\begin{lem} \label{exotic}
Any $C^k$-structure on $\mathbb{R}^2$ is $C^k$-diffeomorphic to the standard $C^k$-structure on $\mathbb{R}^2$.
\end{lem}
\begin{proof}
Whitney proved that if ${\cal{A}}^k$ is a maximal $C^k$-atlas on $\mathbb{R}^2$, then there exists (at least one) maximal $C^\infty$-atlas ${\cal{A}}^\infty \subset {\cal{A}}^k$.  Since our claim is true for $k=\infty$ we know there exist a $C^\infty$-diffeomorphism $h: (\mathbb{R}^2,{\cal{A}}^\infty) \longrightarrow (\mathbb{R}^2,{\cal{S}}^\infty)$.
We claim that $h$ is also a $C^k$-diffeomorphism $h: (\mathbb{R}^2,{\cal{A}}^k) \longrightarrow (\mathbb{R}^2,{\cal{S}}^k)$. 
We have $\tilde{\psi} \in {\cal{S}}^\infty$ if and only if $\tilde{\psi} \circ h \in {\cal{A}}^\infty$.  In particular since $I \in {\cal{S}}^\infty$ it follows that $h \in {\cal{A}}^\infty$.  Hence $h \in {\cal{A}}^k$ and $\psi \in {\cal{S}}^k$ if and only if $\psi \circ h \in {\cal{A}}^k$.
\end{proof}
\subsection{A $C^{1M}$ Hilbert's Theorem}
\begin{thm}\label{globalhw} (Globalizing Theorem \ref{hartwin})
Let $D$ be a rectangle in $\mathbb{R}^2$. Let $f_{asyche}:D \stackrel{C^{1M}}{\longrightarrow} \mathbb{R}^3$ be a regular immersion satisfying all the conditions of Theorem \ref{hartwin}.  Then there exist a $C^1$ diffeomorphism $\rho:D \longrightarrow D$  such that $f_{new} = f_{asyche} \circ \rho : D \longrightarrow \mathbb{R}^3$ is a regular $C^2$-immersion. 
\end{thm}
\begin{proof} 
First recall from Theorem \ref{hartwin} that $f_{graph}= f_{asyche} \circ S: \breve{D} \longrightarrow \mathbb{R}^3$ is $C^2$.  Let $\pi:\mathbb{R}^3 \longrightarrow \mathbb{R}^2$ denote the corresponding projection.  Then $\pi \circ f_{graph} = I$, and $\pi \circ f_{asyche} = S^{-1}$.  
We claim the $\{(S_{\breve{D}^{-1}},\breve{D})\}$ is a set of $C^2$-related charts.  This is because 
\begin{eqnarray}
S_{\breve{D'}}^{-1} \circ S_{\breve{D}} &=& \pi ' \circ f_{asyche} \circ S_{\breve{D}}, \\
&=& {f'}_{graph} \circ S_{\breve{D}}
\end{eqnarray}
is $C^2$.
Let $\mathcal{A}$ be the unique maximal $C^2$ atlas containing all these charts.  By construction, relative to this $C^2$ differentiable structure the map $f_{asyche}:D_{\mathcal{A}} \longrightarrow \mathbb{R}^3$ is a regular $C^2$-immersion.  By Lemma \ref{exotic}$\;D_{\mathcal{A}}$ is $C^2$ diffeomorphic to $D_{\mathcal{S}^2}$ by some (global) $C^2$ diffeomorphism $\rho: D_{\mathcal{S}^2} \longrightarrow D_{\mathcal{A}}$.
\end{proof}
As an immediate consequence of Theorem \ref{globalhw} we obtain
\begin{thm}($C^{1M}$-type Hilbert Theorem)
Let $D$ be rectangle in $\mathbb{R}^2$ and  $f_{asyche}:D \stackrel{C^{1M}}{\longrightarrow} \mathbb{R}^3$ a regular immersion 
satisfying all the conditions of Theorem \ref{globalhw}.  Then the metric induced by $f_{asyche}$ is not complete.
\end{thm}
\begin{proof}
By Theorem \ref{globalhw} the map $f_{new} = f_{asyche} \circ \rho: D \longrightarrow \mathbb{R}^3$ is a regular $C^2$-immersion.
By Efimov's Theorem \cite{E},\cite{KM} the immersion $f_{new}$ induces an incomplete metric. Since $\rho$ is a global diffeomorphism and a reparametization it is an isometry.  Hence $f_{asyche}$ also induces an incomplete metric.
\end{proof}

\section{Standard Cusp Line Example: The PS-sphere} \label{cusppics}
We use the following parametrizations for the pseudo-sphere.  In curvature line coordinates
\[f(u,v) = (\cos u \;\mathrm{sech} \;v, \sin u \;\mathrm{sech} \;v, v - \tanh v ),\]
and in asymptotic coordinates
\[f(x,y) = (\cos(x-y) \;\mathrm{sech}(x+y), \sin(x-y) \; \mathrm{sech}(x+y), (x+y) - \tanh(x+y).\]
Note that $\Vert f_x \Vert = \Vert f_y \Vert = 1$.  We define two normals, both in asymptotic coordinates, as follows:
\[N_{standard} = \frac{\f_x \times \f_y}{\Vert \f_x \times \f_y \Vert} \]
whenever $\Vert \f_x \times \f_y \Vert \neq 0$ (and undefined otherwise),
and
\[N = N_{front} = (\cos(x-y)\; \tanh(x+y),-\sin(x-y) \tanh(x+y),\mathrm{sech}(x+y))\]
which is defined for all $(x,y)$.

Note that $\Vert N \Vert = 1$, $\Vert N_x \Vert = 1$, $\Vert N_y \Vert = 1$, $N \perp f_x$, and $N \perp f_y$.  Furthermore
\[f_x = N \times N_x, f_y = -N \times N_y.\]
We define 
\[f_x^\perp = N \times f_x.\]
Note that ${f_x,f_x^\perp, N}$ is a positively oriented orthonormal frame with \[\det(f_x,f_x^\perp,N)=1.\]
This allows us to unambiguously define the oriented angle $\omega$ from $f_x$ to $f_y$ in the the oriented plane spanned by $f_x$ and $f_x^\perp$.
\[\omega = \angle(\f_x,\f_y).\]
Note that
\[\omega = 4 \tan^{-1} (e^{x+y}).\]
Note that $0<\omega<2 \pi$, $0<\frac{\omega}{2} < \pi$, and $0<\frac{\omega}{2} + \frac{\pi}{2} < 1.5 \pi < 2 \pi$.
Furthermore $\omega_x = 2\; \mathrm{sech}(x+y)$, $\omega_y = 2\; \mathrm{sech}(x+y)$, $\omega_{xy} = -2\; {\mathrm{sech}(x+y)} {\tanh(x+y)}$, and $\sin(\omega) = -2\;  \mathrm{sech}(x+y) \tanh(x+y)$.  In particular
\[\omega_{xy} = \sin(\omega).\]
Similarly we have
\[N_{xy} = \cos(\omega) N, \; f_{xy} = \sin(\omega) N.\]
Whenever $\omega \neq \pi$ (recall $0 < \omega < 2 \pi$) we have $N_{standard}$ is defined and 
\[N_{front} =  \frac{\sin \omega}{|\sin \omega|}  N_{standard}.\]
We define the globally positively oriented orthonormal curvature line  frame $e_1,e_2,N$ along asymptotic lines by
\begin{eqnarray*}
e_1 &=& \cos \frac{\omega}{2} f_x + \sin \frac{\omega}{2} f_x^\perp \\
e_2 &=& \cos (\frac{\omega}{2}+\frac{\pi}{2}) f_x + \sin (\frac{\omega}{2}+\frac{\pi}{2}) f_x^\perp.
\end{eqnarray*}
Note: $\det(e_1,e_2,N) =1$.  Figures \ref{cusp2}-\ref{fence} below show how the frame $e_1,e_2,N$ and in particular the front $N$ behave smoothly along the asymptotic curves even as it passes through the ``singular'' cusp line.

\begin{figure}
\hspace{-.25in}\includegraphics[scale=.5]{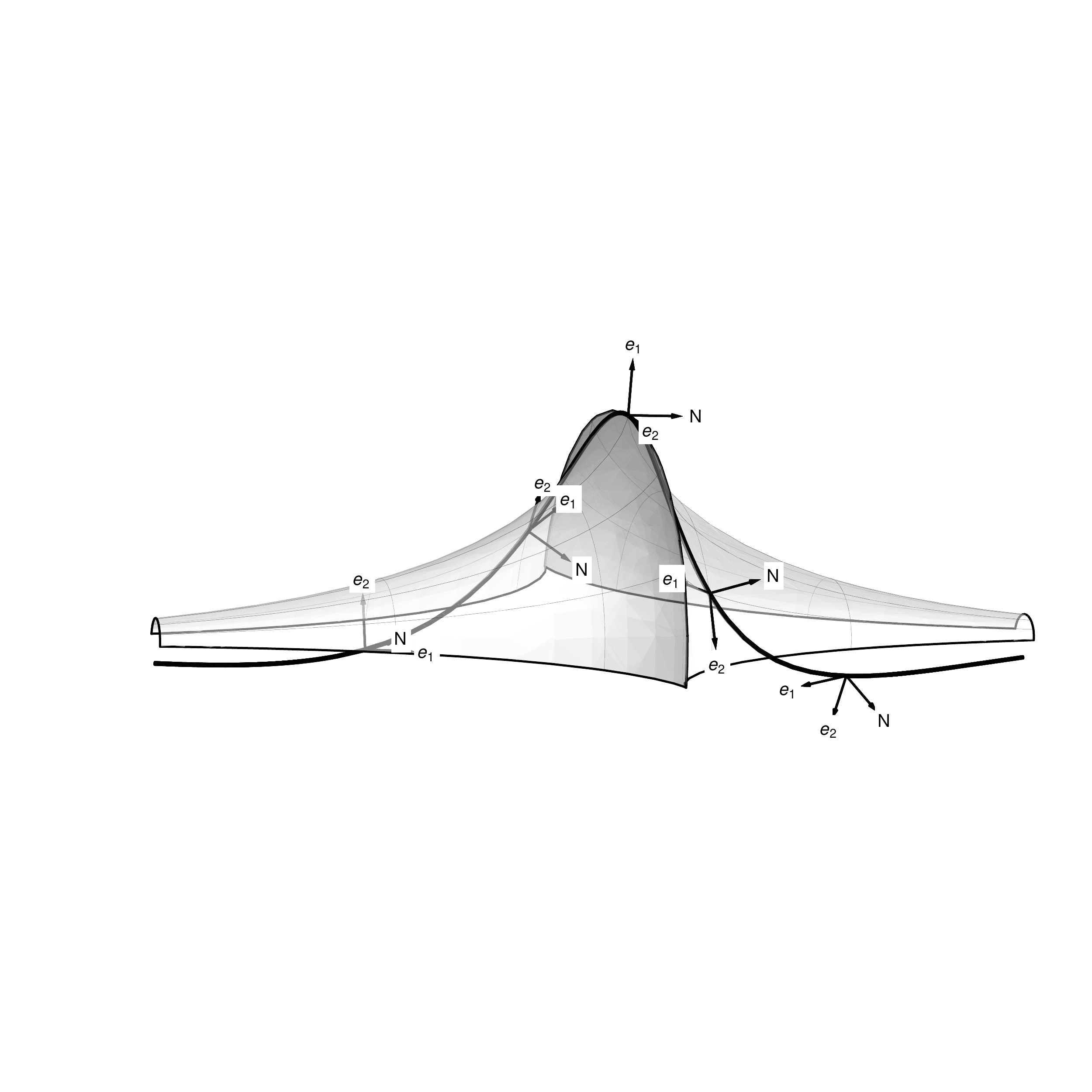}
\caption{Orthonormal Frame $(e_1,e_2,N)$ along an Asymptotic Curve on Surface}\label{cusp2}
\end{figure}
 
\begin{figure}
\begin{center}
\hspace{-.25in}\includegraphics[scale=.5]{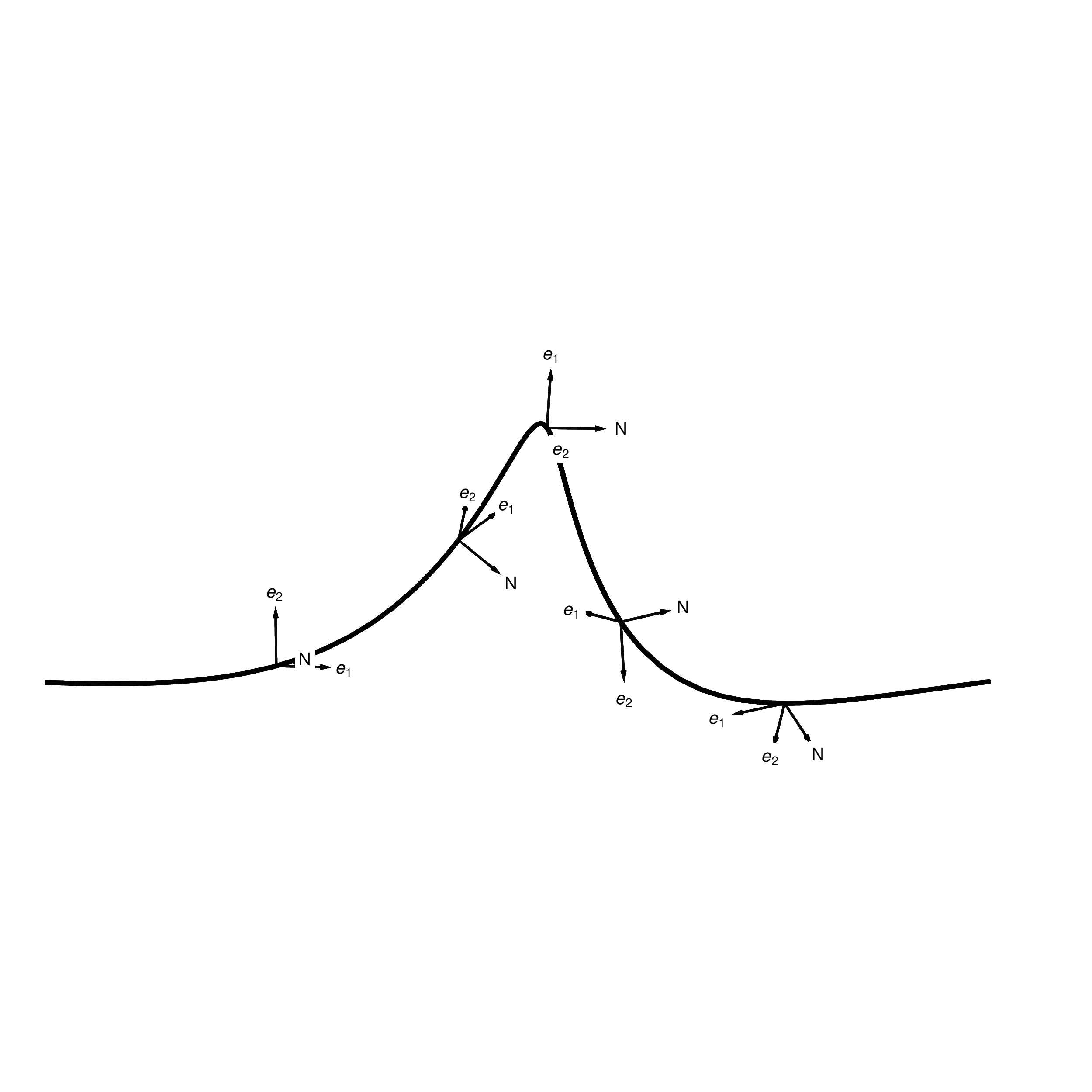}
\caption{The $(e_1,e_2,N)$ Frame Behaves Smoothly along an Asymptotic Curve}\label{curve2}
\end{center}
\end{figure}

 \begin{figure}
\begin{center}
\hspace{-.25in}\includegraphics[scale=.5]{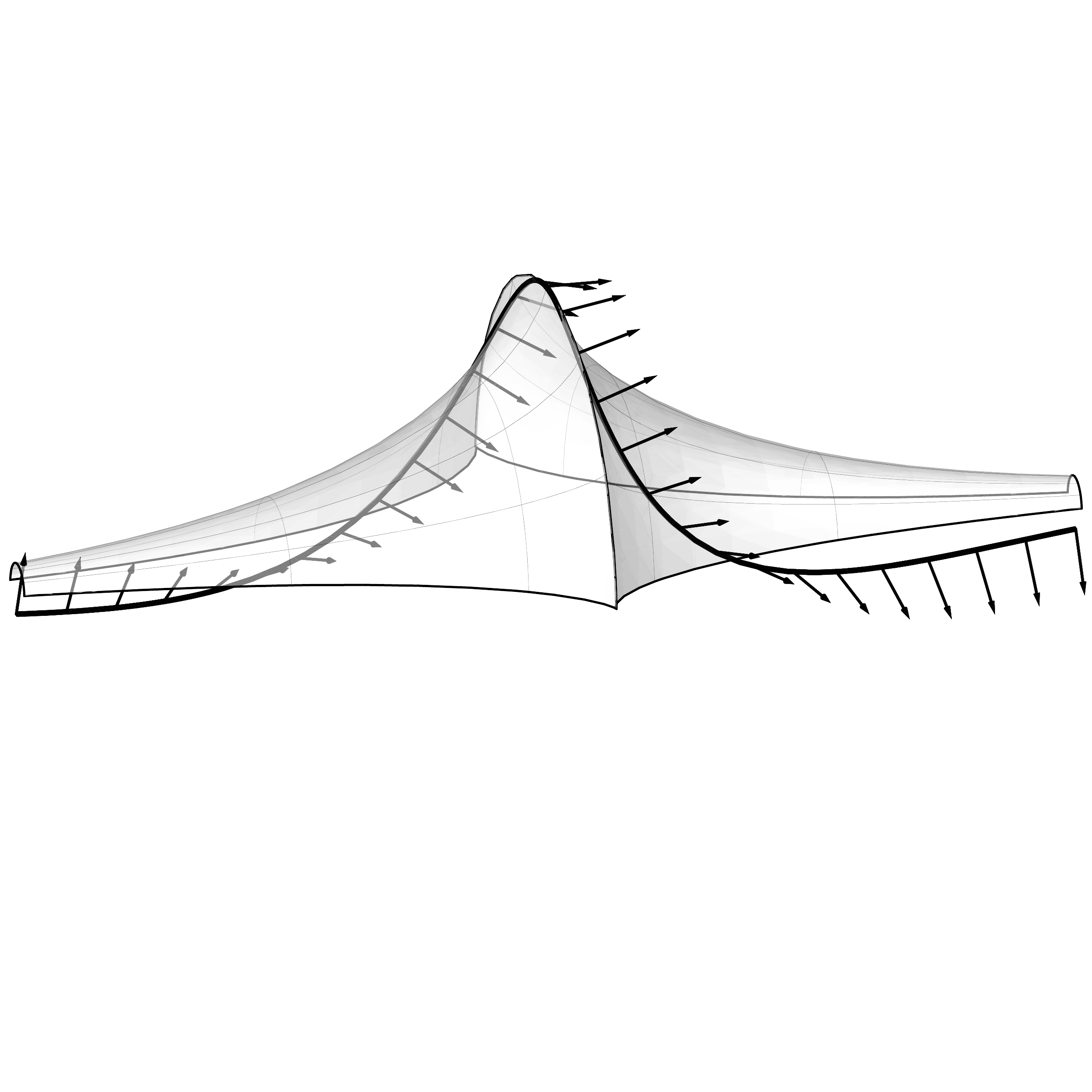}
\caption{The Front N along an Asymptotic Coordinate on the Surface}\label{fenceon}
\end{center}
\end{figure}
 
\begin{figure}
\begin{center}
\hspace{-.25in}\includegraphics[scale=.5]{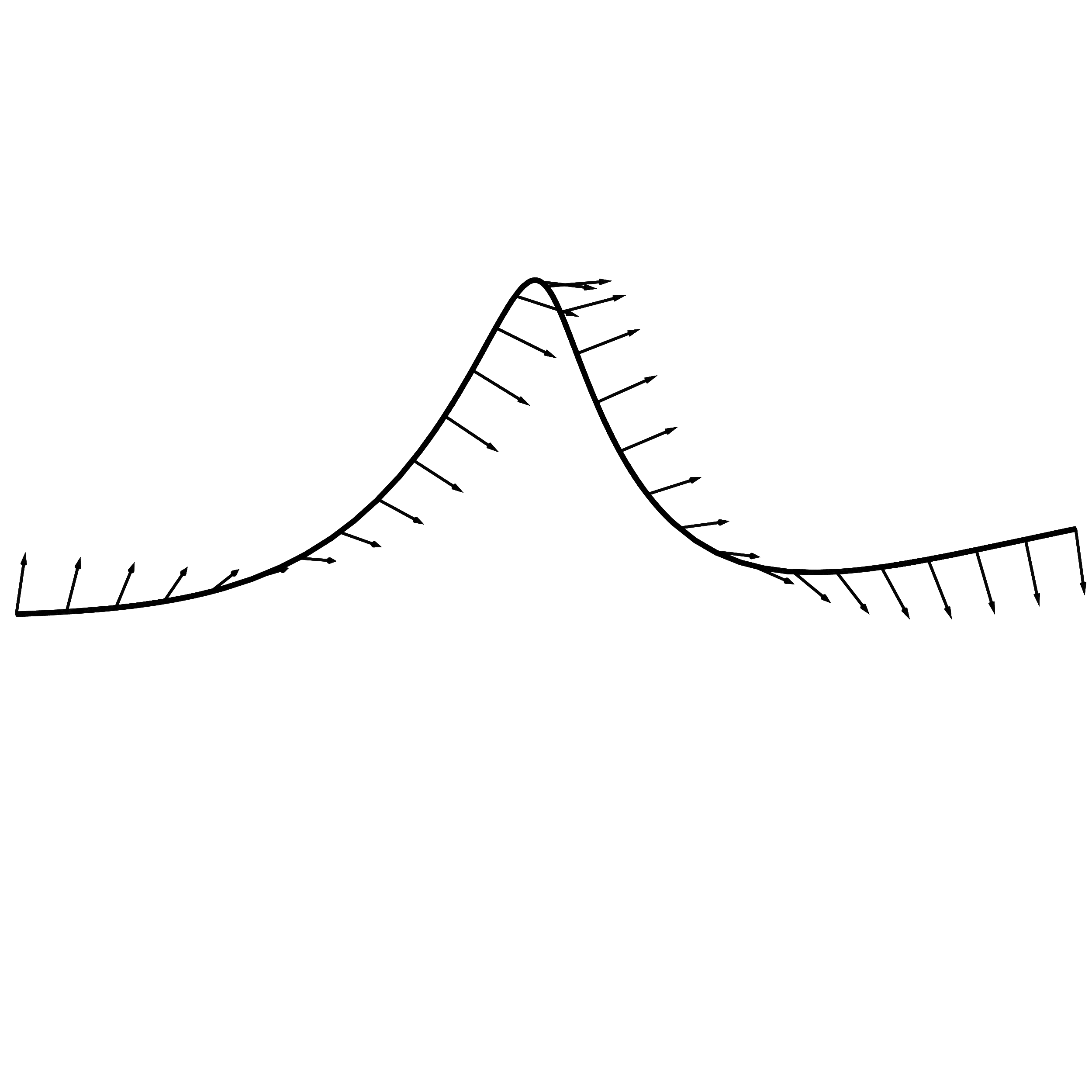}
\caption{The Front N has Nice Smooth Behavior in Asymptotic Coordinates}\label{fence}
\end{center}
\end{figure}

\end{document}